\documentclass[11pt,leqno]{amsart}
\usepackage[total={6in,9in},
top=1in, left=1in, right=1in, bottom=1in]{geometry}
\usepackage{amsmath}
\usepackage[italian, english]{babel}
\usepackage{amsfonts}
\usepackage{amssymb}
\usepackage{dsfont}
\usepackage{mathrsfs}
\usepackage{graphicx}
\usepackage{fancyhdr}
\usepackage{multirow}
\usepackage{multicol}
\usepackage{enumitem}
\usepackage{amsthm}
\usepackage{empheq}
\usepackage{cases}
\usepackage[all]{xy}
\usepackage{stmaryrd}
\usepackage[colorlinks, citecolor=blue, linkcolor=red]{hyperref}
\usepackage{mathrsfs}
\usepackage{tikz,tikz-cd}
\def\SS{{{\mathbb S}}}

\def\RR{{\mathbb R}}

\def\gt{\widetilde{g}}

\tikzset{
	subset/.style={
		draw=none,
		edge node={node [sloped, allow upside down, auto=false]{$\subset$}}},
	Subset/.style={
		draw=none,
		every to/.append style={
			edge node={node [sloped, allow upside down, auto=false]{$\subset$}}}
	}
}
\tikzset{
	labl/.style={anchor=south, rotate=90, inner sep=.50mm}
}

\newcommand{\erre}{\mathds{R}}

\newcommand{\cinf}{C^{\infty}(M)}

\newcommand{\partderf}[2]{\frac{\partial {#1}}{\partial {#2}}}

\newcommand{\christ}{\Gamma_{ij}^k}
\newcommand{\riem}{\operatorname{Riem}}
\newcommand{\ricc}{\operatorname{Ric}}
\newcommand{\weyl}{\operatorname{W}}

\newcommand{\hess}{\operatorname{Hess}}

\newcommand{\KN}{\mathbin{\bigcirc\mspace{-15mu}\wedge\mspace{3mu}}}

\newcommand{\ra}{\rightarrow}

\newcommand{\set}[1]{{\left\{#1\right\}}}               
\newcommand{\pa}[1]{{\left(#1\right)}}                  
\newcommand{\sq}[1]{{\left[#1\right]}}                  
\newcommand{\abs}[1]{{\left|#1\right|}}                 

\newcommand{\eps}{\varepsilon}                           

\newcommand{\ol}[1]{\overline{#1}}

\renewcommand{\tilde}[1]{\widetilde{#1}}

\newcommand{\ric}{\mathrm{Ric}}
\newcommand{\bach}{\operatorname{B}}
\def\gt{\widetilde{g}}
\def\gb{\overline{g}}

\newtheorem{theorem}{\textbf{Theorem}}[section]
\newtheorem{lemma}[theorem]{\textbf{Lemma}}

\newtheorem{cor}[theorem]{\textbf{Corollary}}

\theoremstyle{remark}
\newtheorem{rem}[theorem]{\textbf{Remark}}

\numberwithin{equation}{section}
\allowdisplaybreaks

\title[]
{Bach-pinched metrics on closed manifolds}


\subjclass[2010]{}

\begin{document}
	\maketitle
	
	\begin{center}
		\textsc{\textmd{L. Branca \footnote{Universit\`{a} degli Studi di Milano, Italy.
					Email: letizia.branca@unimi.it.}, G. Catino \footnote{Politecnico di Milano, Italy.
					Email: giovanni.catino@polimi.it.}, D. Dameno \footnote{Universit\`{a} degli Studi di Torino, Italy.
					Email: davide.dameno@unito.it.}.}}
	\end{center}
	\begin{abstract}
		Exploiting the deformation method introduced by Aubin in his seminal work to construct 
		constant negative scalar curvature metrics, we show the existence, 
		on every closed manifold of dimension four, of a metric whose 
		Bach tensor is pinched by the scalar curvature. 
	\end{abstract}

\
	
	\section{Introduction}
	Let $(M,g)$ be a Riemannian manifold of dimension $n\geq3$: it is well known that the Riemann curvature tensor $\operatorname{Riem}_g$ admits the
	decomposition
	\[
	\operatorname{Riem}_g=\weyl_g+\dfrac{1}{2(n-1)}
	\ricc_g\KN g
	-\dfrac{S_g}{2(n-1)(n-2)}g\KN g,
	\]
	where $\weyl_g$ is the \emph{Weyl tensor},
	$\ricc_g$ is the \emph{Ricci tensor} and
	$S_g$ is the \emph{scalar curvature} and $\KN$ denotes the
	Kulkarni-Nomizu product. A fundamental question in Riemannian Geometry  is to understand the relations between the curvature and the topology of the underlying manifold: for instance, an example of this relation is provided by metrics with {\em positive scalar curvature} \cite{GromovLawson1,GromovLawson2,LeBrunKdim,SchoenYau}, {\em non-positive sectional curvature} or by metrics which are {\em locally conformally flat}, i.e. $\weyl_g\equiv0$ for $n\geq4$ \cite{Avez,Goldman,Kuiper}.\\
	On the other hand, there are examples of curvature conditions that are unobstructed: in \cite{Aubin1,Aubin2} Aubin showed that, on every smooth $n$-dimensional closed (compact with empty boundary) manifold, there exists a smooth Riemannian metric with {\em constant negative scalar curvature}. This result was then extended to complete non-compact manifolds by  Bland and Kalka in \cite{BlandKalka}. In particular, there are no topological obstructions to metrics with negative scalar curvature and, more in general, Lohkamp proved that on every smooth complete Riemannian manifold there are no obstructions to the existence of metrics with {\em negative Ricci curvature} \cite{Lohkamp}.\\
	Note that in \cite{Aubin2} Aubin also proved that, if $M$ is a closed Riemannian manifold of dimension $n\geq4$, then there always exists a metric with non-vanishing Weyl curvature, i.e. $\abs{\weyl_g}_g>0$ everywhere. As a consequence, in \cite{CatinoMastroliaMonticelliPunzo} the authors showed the existence of \emph{weak harmonic Weyl} metrics on every closed Riemannian four-manifold, i.e. critical points, in a conformal class, of the normalized $L^2$-norm of the Cotton tensor. Moreover, in \cite{Catino} the second author extended Aubin's construction of metrics with negative scalar curvature proving that every $n$-dimensional closed manifold admits a Riemannian metric with {\em constant negative scalar-Weyl curvature}, i.e.
	\[S_g+t\abs{\weyl_g}_g\equiv-1,\]
	for every $t\in\erre$ and, as a consequence, there are no topological obstructions to the existence of metrics with negative scalar curvature and {\em Weyl-pinched curvature}, generalizing a result of Seshadri in dimension four \cite{Seshadri}: namely, on every closed manifold of dimension $n\geq4$, for every $\eps>0$ there exists a Riemannian metric $g=g_{\eps}$ such that
	\begin{align*}
		S_g<0\quad\text{and }\abs{\weyl_g}_g^2<\eps S_g^2\text{ on }M.
	\end{align*}
	In this paper, we are interested in the Bach tensor $\bach_g$, 
	which, if $n=4$, is locally defined as
	\[B_{ij}=W_{ikjl,lk}+\frac{1}{2}R_{kl}W_{ikjl}.\]
	This geometric quantity, introduced by Bach in the
	context of Conformal Relativity \cite{Bach1921}, is a divergence-free, conformally 
	covariant tensor, i.e., if $\tilde{g}=e^{2u}g$ is a metric conformal to 
	$g$, $\bach_g$ satisfies \cite{derd}
	\begin{equation} \label{bachconfcov}
	\bach_{\tilde{g}}=e^{-2u}\bach_g.
	\end{equation}
	We say that a Riemannian metric $g$ is \emph{Bach-flat} if $\bach_g\equiv 0$ on $M$: typical examples are provided by conformally
	Einstein metrics, locally conformally flat metrics, and, more generally, half conformally flat metrics.
	Furthermore, Bach-flat metrics are exactly the 
	critical points of the so-called Weyl functional
	\[\mathfrak{W}(g)=\int_M\abs{\weyl_g}_g^2dV_g.\]
	
	Up to now, no topological obstructions to the
	existence of Bach-flat metrics on closed smooth four-manifolds
	are known, although
	some partial rigidity and classification results have been
	proven \cite{mantegazza, caochen, changguryang, lebrunbach, tianviac2, tianviac}: we recall that one of the few examples of non-trivial
	Bach-flat metrics was constructed by Abbena, Garbiero and Salamon on
	a solvable Lie group \cite{abbena}. In \cite{gurskyviac}, Gursky
	and Viaclovsky developed a gluing method to construct 
	$B^t$-flat metrics, i.e. critical points of the
	Weyl functional perturbed with a quadratic scalar curvature term,
	on connected sums of Einstein four-manifolds.
	On the other hand, it is possible
	to find unobstructed metrics whose Bach tensor satisfies
	curvature properties: indeed, 
	using a special metric deformation introduced by Aubin \cite{Aubin2}, 
	the second and the third author, together with P. Mastrolia,
	showed the existence of metrics with non-vanishing Bach tensor 
	on every four-dimensional closed manifold \cite{CDM}.
	
	Our main result is the following
	
	\begin{theorem}\label{t-main} On every smooth $4$-dimensional closed manifold $M$, for every $t\in\RR$, there exists a smooth Riemannian metric $g=g_{t}$ with 
		$$
		S_g+t|\bach_g|^{\frac{1}{2}}_g\equiv -1\quad\text{on } M.
		$$
		In particular, there are no topological obstructions for negative scalar-Bach curvature metrics.
	\end{theorem}
	Therefore, choosing $t=1/\sqrt{\eps}$, $\eps>0$,  in Theorem \ref{t-main} we obtain the following existence result for metrics with negative scalar curvature and {\em Bach-pinched curvature}:	
	\begin{cor} On every smooth $n$-dimensional closed manifold, for every $\varepsilon>0$, there exists a smooth Riemannian metric $g=g_{\varepsilon}$ with 
		$$
		S_g<0 \quad\text{and}\quad |\bach_g|_g <\varepsilon S_g^2\quad\text{on } M.
		$$
	\end{cor}
	The naive idea of the proof of Theorem \ref{t-main}, similar
	to the one exploited in \cite{Aubin1, Aubin2, Catino}, would be 
	to start from a reference metric $g$, to be suitably chosen, 
	to construct a non-conformal Riemannian metric $\gb$
	such that 
	\[\int_M \pa{S_{\gb}+t\abs{\bach_{\gb}}_{\gb}^{\frac{1}{2}}}dV_{\gb}<0\] 
	and to apply the
	method used by
	Gursky in \cite{gur1} to produce the desired constant negative scalar-Bach metric.
	Since we must deal with the modified conformal Laplacian of $g$
	\[
	-6\Delta_g+S_g+t\abs{\bach_g}_g^{\frac{1}{2}},
	\]
	to construct smooth (at least, $C^4$) metrics
	we have to take into account the lack of regularity of the operator
	on $\set{\bach_g=0}$: to overcome this difficulty, throughout the proof we often rely on
	the existence of Riemannian metrics with non-vanishing
	Bach tensor \cite{CDM} (see Section \ref{proofth} for further details).
	We also want to stress out the fact that, in order to have precise
	estimates, we had to compute
	the full variation formula of the Bach tensor under Aubin's
	deformation (see Section \ref{app} for all the detailed computations).

\

	\section{The scalar-Bach curvature}
	In this section we focus on the variational and conformal aspects of the scalar-Bach curvature, which are analogous to those of the scalar-Weyl curvature, first studied by Gursky in \cite{gur1}. Let $(M, g)$ be a $n$-dimensional closed (compact with empty boundary) Riemannian manifold. We start by recalling the definition of the conformal Laplacian is the operator:
	$$
	\mathcal{L}_g:=-\frac{4(n-1)}{n-2}\Delta_g+S_g.
	$$
	It satisfies the following well known conformal covariance property: if $\gt=u^{4/(n-2)} g$, where
	$u$ is a positive smooth function on $M$, then 
	$$
	\mathcal{L}_{\gt}\phi = u^{-\frac{n+2}{n-2}}\mathcal{L}_g(\phi u),\quad \forall \phi\in C^2(M).
	$$
	Observe that this operator plays a prominent role in the resolution of the Yamabe variational problem: indeed, the scalar curvature of the conformally related metric $\gt$ is given by
	$$
	S_{\gt}=u^{-\frac{n+2}{n-2}} \mathcal{L}_g u.
	$$
	In \cite{gur1}, Gursky introduced a modification of the conformal Laplacian, introducing a new term depending on the Weyl curvature. Given $t\in\RR$, we recall the definition of scalar-Weyl curvature
	\begin{equation}\label{scalar-weyl}
		F_g:= S_g + t \abs{\weyl_g}_g
	\end{equation}
	and the associated modified conformal Laplacian
	$$
	\mathcal{L}^t_g:=-\frac{4(n-1)}{n-2}\Delta_g+F_g,
	$$
	where
	$$\abs{\weyl_g}_g=\sqrt{W_{ijkl}W_{pqrs}g^{ip}g^{jq}g^{kr}g^{ls}}$$
	denotes the norm of the Weyl curvature of $g$.	
	It was proved in \cite{gur1} that the pairs $(F_g, \mathcal{L}^t_g)$ and $(S_g, \mathcal{L}_g)$ share the same conformal properties. In fact, if $\gt=u^{4/(n-2)} g$, then
	\begin{equation}\label{e-confFWeyl}
		\mathcal{L}^t_{\gt}\phi = u^{-\frac{n+2}{n-2}}\mathcal{L}^t_g(\phi u),\quad \forall \phi\in C^2(M),\quad\text{and}\quad F_{\gt}=u^{-\frac{n+2}{n-2}} \mathcal{L}^t_g u.
	\end{equation}
	In an analogous way, for $n=4$, given $t\in\RR$ we define the \emph{scalar-Bach curvature}
		\begin{equation}\label{scalar-bach}
			F_g^B:= S_g + t \abs{\bach_g}^{\frac{1}{2}}_g
		\end{equation}
	and the associated modified conformal Laplacian 
		$$
		\mathscr{L}^t_g:=-6\Delta_g+F^B_g,
		$$
	where $\abs{\bach_g}_g=\sqrt{B_{ij}B_{pq}g^{ip}g^{jq}}$ denotes the norm of the Bach tensor of $g$. A crucial observation is the fact that the pair $(F^B_g,\mathscr{L}_g^t)$ preserves the same conformal properties of $(S_g,\mathcal{L})$; indeed, let $\gt=u^{2}g$, then 
	$$u^{2}(\bach_{\tilde{g}})_{ij}=(\bach_g)_{ij}$$
	and 
	$$u^{8}\abs{\bach_{\tilde{g}}}_{\tilde{g}}^2=\abs{\bach_g}_g^2,$$
	therefore
	\begin{equation}\label{e-confF}
		\mathscr{L}^t_{\gt}\phi = u^{-3}\mathscr{L}^t_g(\phi u),\quad \forall \phi\in C^4(M),\quad\text{and}\quad F^B_{\gt}=u^{-3} \mathscr{L}^t_g u.
	\end{equation}
	In particular, adapting the argument of \cite[Proposition 3.2]{gur1}, we have the following: 
	\begin{lemma}\label{l-g1} Let $(M,g)$ be a $4$-dimensional closed Riemannian manifold with $\abs{\bach_g}_g>0$. Then, there exists  a smooth metric $\gt\in[g]$ with either $F^B_{\gt}>0$, $F^B_{\gt}<0$, or $F^B_{\gt}\equiv 0$. Moreover, these three possibilities are mutually exclusive.
	\end{lemma}
	\begin{proof}
		Let $\mu_t(g)$ denote the principle eigenvalue of $\mathscr{L}^t_g$ and let $\phi$ denote the eigenfunction relative to $\mu_t(g)$. By the maximum principle $\phi$ can be assumed to be positive. In particular $\phi$ satisfies
		$$
		\mathscr{L}^t_g\phi=\mu_t(g)\phi,
		$$
		that is equivalent to
		$$
		-6\Delta\phi=-F_g^B\phi+\mu_t(g)\phi.
		$$
		Note that, since $\abs{\bach_g}_g> 0$, then $F_g^B\in C^{\infty}(M)$ and thus $\phi\in C^{\infty}(M)$. 
		Let us consider the conformal change $\gt=\phi^2g$, then $\gt\in[g]$ is smooth and by
		\eqref{e-confF}
		$$
		F_{\gt}^B=\mu_t(g)\phi^{-2}.
		$$ 
		Therefore, $F^B_g$ is either positive, negative or identically zero, depending on the sign of $\mu_t(g)$ and these possibilities are mutually exclusive because the sign of $\mu_t(g)$ is conformally invariant.
	\end{proof}
	\noindent
	In analogy with the Yamabe problem, Gursky introduced the following modified fuctional
	$$
	\widehat{Y}(u):=\frac{\int_M u\, \mathcal{L}^t_g u\,dV_g}{\left(\int_M u^{2n/(n-2)}\,dV_g\right)^{(n-2)/2}}
	$$
	and 
	$$
	\widehat{Y}(M,[g]):=\inf_{u\in H^1(M)} \widehat{Y}(u),
	$$
	which is conformally  invariant.
	Following a classical subcritical regularization argument, he proved that, if $\widehat{Y}(M,[g])\leq 0$, then the variational problem of finding a conformal metric $\gt\in[g]$ with constant scalar-Weyl curvature $F$ can be solved. See \cite[Proposition 3.5]{gur1} for the proof (in dimension four). In an analogous way, when $n=4$ and $\abs{\bach_g}_g> 0$, we can consider the functional   
	$$
	\widehat{Y}^B(u):=\frac{\int_M u\, \mathscr{L}^t_g u\,dV_g}{\left(\int_M u^{4}\,dV_g\right)}
	$$
	and the conformal invariant
	$$
	\widehat{Y}^B(M,[g]):=\inf_{u\in H^1(M)} \widehat{Y}^B(u).
	$$
	By \eqref{e-confF}, it is easy to see that the functional $u\mapsto\widehat{Y}^B(u)$ is equivalent to the modified Einstein-Hilbert functional
	$$
	\gt=u^{2}g \longmapsto \frac{\int_M F^B_{\gt}\,dV_{\gt}}{\text{Vol}_{\gt}(M)}.
	$$
	In particular the following lemma holds:	
	\begin{lemma}\label{l-g2} Let $(M,g)$ be a $4$-dimensional closed Riemannian manifold with $\abs{\bach_g}_g>0$. If there exists a metric $g'\in[g]$ such that
		$$
		\int_M F_{g'}\,dV_{g'} <0,
		$$
		then, there exists a (unique) $C^{\infty}$ metric $\gt\in[g]$ such that $F_{\gt}\equiv -1$.
	\end{lemma}
	\begin{proof}
		Since 
		$$
		\int_M F_{g'}\,dV_{g'} <0,
		$$
		arguing as in \cite[Proposition 4.4]{LeeParker}, there exists $v\in H^1(M)$ which attains the minimum of $\widehat{Y}^B(M,[g])$. 
		In particular, 
		$$
		-6\Delta v+F_g^Bv=Kv^{-3};
		$$
		where $K$ is a negative constant. Then, since $\abs{\bach_g}_g>0$, $F_g$ is smooth and, by elliptic
		regularity, we have that $v\in\cinf$ and $\gt\in[g]$ such that $\gt=v^{-2}g$ is smooth.
	\end{proof}
	Note that these techniques introduced by Gursky have been used in various contexts, such as \cite{BCDM, Catino, GurskyLeBrun1, GurskyLeBrun2, Itoh, LeBrun}.	
	
\
	
	\section{Aubin's metric deformation}\label{sec-3}
	We recall the deformation of a Riemannian metric $g$, introduced by 
	Aubin in \cite{Aubin1, Aubin2} and defined as
	\begin{equation} \label{aubindef}
		\ol{g}=g+df\otimes df, \quad{ } f\in\cinf;
	\end{equation}
	throughout this section, the barred 
	quantities are referred to the metric $\ol{g}$, while the 
	unbarred ones are related to $g$. 
	Locally, given an open chart $U\subset M$ with
	coordinate functions $dx^1,...,dx^n$, we can rewrite 
	\eqref{aubindef} as
	\begin{equation}
		\ol{g}_{ij}=g_{ij}+f_if_j,
	\end{equation}
	where $f_i=\partial_if=\partderf{f}{x^i}$. The 
	Levi-Civita connection is locally expressed by the 
	Christoffel symbols $\christ$, which, with
	respect to $\ol{g}$, are defined as
	\begin{equation} \label{christaubin}
		\ol{\Gamma}_{ij}^l=\christ+\dfrac{f^lf_{ij}}{1+\abs{\nabla f}^2},
	\end{equation}
	where $f^i=g^{ij}f_j$, $f_{ij}=\partial_jf_i-\christ f_l$
	and $\christ$ are the Christoffel symbols of the metric $g$. 
	In particular, we have
	\begin{align}\label{volume form and inverse aub def}
		dV_{\ol{g}}&=\pa{1+\abs{\nabla f}^2}^{\frac{1}{2}}dV_g;\\
		\ol{g}^{ij}&=g^{ij}-\frac{f_if_j}{1+\abs{\nabla f}^2}.\notag
	\end{align}
	Similarly, starting from 
	\eqref{christaubin}, we can compute the curvature components of the metric 
	$\ol{g}$ (see \cite[Chapter 2]{CMbook}): for instance, the local components of the $(0,4)$-Riemann tensor
	$\ol{\riem}$ are written as 
	\begin{equation} \label{riemaubin}
		\ol{R}_{ijlt}=R_{ijlt}+\dfrac{1}{1+\abs{\nabla f}^2}\pa{
			f_{il}f_{jt}-f_{it}f_{jl}}=:R_{ijkl}+E^R_{ijkl}.
	\end{equation}
	Tracing \eqref{riemaubin}, we obtain the local expressions for the 
	Ricci tensor $\ol{\ric}$ and the scalar curvature $\ol{S}$:
	\begin{align}
		\ol{R}_{ij}&=R_{ij}-\dfrac{1}{1+\abs{\nabla f}^2}f^tf^lR_{itjl}
		+\dfrac{1}{1+\abs{\nabla\phi}^2}\pa{\Delta f\cdot f_{ij}-
			f_{it}f_j^{t}}+ \label{ricciaubin}\\
		&-\dfrac{1}{\pa{1+\abs{\nabla f}^2}^2}f^tf^l
		\pa{f_{ij}f_{tl}-f_{it}f_{jl}}\notag\\
		&=:R_{ij}+F_{ij}\notag;\\
		\ol{S}&=S-\dfrac{2}{1+\abs{\nabla f}^2}R_{ij}f^if^j+
		\dfrac{1}{1+\abs{\nabla f}^2}\sq{\pa{\Delta f}^2-
			\abs{\hess f}^2}+ \label{scalaubin}\\
		&-\dfrac{2}{\pa{1+\abs{\nabla f}^2}^2}
		\sq{\Delta f\cdot f^if^jf_{ij}-
			f^if_{ij}f^{jp}f_p}\notag\\
		&=:S+H	\notag.
	\end{align}
	Note that the proof of \eqref{volume form and inverse aub def} and that of the scalar curvature can be found in \cite{Aubin2} while the other transformations can be found in \cite[Chapter 2]{CMbook}. Moreover, on a $4$-dimensional manifold, we have
	\begin{align}\label{bachdef}
		B_{ij}=&\frac{1}{2}\sq{\Delta R_{ij}-\frac{1}{3}S_{ij}+2R_{kl}R_{ikjl}
			-\frac{2}{3}SR_{ij}-\frac{1}{6}\Delta S g_{ij}-\frac{1}{2}\pa{\abs{\ric}^2-\frac{S^2}{3}}g_{ij}}.
	\end{align}
	Then
	\begin{equation}\label{bach aubin}
		\ol{B}_{ij}=B_{ij}+E(f)_{ij},
	\end{equation}
	where
	\begin{align}\label{resto bach}
		E(f)_{ij}:=&\frac{1}{2}\Bigg[\bar{\Delta} \ol{R}_{ij}-\frac{1}{3}\ol{S}_{ij}+2\ol{R}_{kl}\ol{R}_{ikjl}
			-\frac{2}{3}\ol{S}\ol{R}_{ij}-\frac{1}{6}\bar{\Delta} \ol{S} \ol{g}_{ij}-\frac{1}{2}\pa{\abs{\ol{\ric}}_{\ol{g}}^2-\frac{\ol{S}^2}{3}}\ol{g}_{ij}\\
			&-\Delta R_{ij}-\frac{1}{3}S_{ij}+2R_{kl}R_{ikjl}
				-\frac{2}{3}SR_{ij}-\frac{1}{6}\Delta S g_{ij}-\frac{1}{2}\pa{\abs{\ric}^2-\frac{S^2}{3}}g_{ij}\Bigg].\notag
	\end{align}
	Moreover, we point out that
	$$
	\ol{S} = S - \frac{R_{ij}f^if^j}{1+\abs{\nabla f}^2} + \nabla^i \left( \frac{\Delta f f_i-f_{ij}f^j}{1+\abs{\nabla f}^2}\right)
	$$
	and thus
	$$
	\int_M \ol{S}\,dV_g = \int_M S\,dV_g - \int_M  \frac{R_{ij}f^if^j}{1+|\nabla f|^2}\,dV_g.
	$$
	We now prove the validity of the following integral sufficient condition for the existence of a constant negative scalar-Bach curvature, in the conformal class $[g]$ of a metric $g$:
	\begin{lemma}\label{l-a1} Let $M$ be a $4$-dimensional closed manifold. If there exists a positive smooth function $u\in C^{\infty}(M)$ such that for a Riemannian metric $g$ on $M$, satisfying $\abs{B_g}_g>0$, it holds
		$$
		\int_M F^B_g\,u^2\,dV_g + 6\int_M |\nabla u|^2\,dV_g <0,
		$$
		then there exists a (unique) $C^{\infty}$ metric $\gt\in[g]$ such that $F^B_{\gt}\equiv -1$.
	\end{lemma}
	\begin{proof} Arguing as in \cite[Lemma 3.2]{Catino}, we consider the conformal metric $g'_{ij}=u^2 g$. By \eqref{e-confF} we have
		$$
		F^B_{g'}=S_{g'}+t|\bach_{g'}|^{\frac{1}{2}}_{g'}= u^{-2}\left(S_g+t|\bach_g|^{\frac{1}{2}}_g-6\frac{\Delta u}{u}\right).
		$$
		Therefore, since $dV_{g'}=u^{4}dV_g$, using the assumption we obtain
		$$
		\int_M F^B_{g'}\,dV_{g'} = \int_M F^B_g\,u^2\,dV_g + 6\int_M |\nabla u|^2\,dV_g <0.
		$$
		The conclusion follows now by Lemma \ref{l-g2}.
	\end{proof}

\noindent
Adapting the method described in \cite{Catino} and
using Lemma \ref{l-a1}, we are able
to find a sufficient condition for the existence
of metrics with constant negative scalar-Bach curvature:
\begin{lemma}\label{l-a2} Let $(M,g)$ be a $4$-dimensional closed manifold. Suppose that there exists a smooth function $f\in C^{\infty}(M)$ such that, for some $t>0$, it holds
	\begin{equation*}
		\int_M \left(S_g+t\abs{\bach_g+E_g(f)}_f^{\frac{1}{2}}\right)\,dV_g -\int_M \frac{R_{ij}f^if^j}{1+\abs{\nabla f}^2}\,dV_g +\frac{3}{2}\int_M \left[\frac{f_{ip}f^pf_{iq}f^q}{(1+\abs{\nabla f}^2)^2} -\frac{\abs{f_{ij}f^if^j}^2}{(1+\abs{\nabla f}^2)^3} \right]\,dV_g<0,
	\end{equation*}
	where $\abs{\cdot}_f$ denotes the norm with respect of $g+df\otimes df$, $E_g(f)$ is defined as in 
	\eqref{resto bach} and 
	\[\abs{\bach_{g+df\otimes df}}_f=\abs{\bach_g+E_g(f)}_f>0, \quad \text{on } M.\] 
	Then, there exists a (unique) $\cinf$ metric $\gt\in[g+df\otimes df]$ such that $F_{\gt}\equiv -1$.
\end{lemma}

\

	\section{Proof of Theorem \ref{t-main}} \label{proofth}
	This section is dedicated to the proof of Theorem \ref{t-main}. We point out that that the technique we use takes strong inspiration from \cite{Aubin1}, \cite{Aubin2} and \cite{Catino}. Before
	we begin the proof, we state the following useful
	\vspace{0.2cm}
	\subsection*{Step 1.} First, we focus on the case
	$$
	t>0.
	$$
	From \cite{CDM}, we know that there exist a metric $g_0$ such that $\abs{\bach_{g_0}}_{g_0}>0$ on $M$, hence we can choose a metric $g\in[g_0]$ such that
	\begin{equation} \label{fpositive}
	F_g^B\geq0\quad\text{on }M,
	\end{equation}
	otherwise, Lemma \ref{l-g1} would imply the existence of a smooth metric $g\in[g_0]$ such that $F_g^B<0$ and Theorem \ref{t-main} would follow from Lemma \ref{l-g2}.
	Consider a positive smooth function $\psi\in C^\infty (M)$, a positive constant $k>0$ and define
	$$
	g' := \psi g,\quad g'' := g' +d(k\psi)\otimes d(k\psi).
	$$
	To prove the existence of a metric $\tilde{g}$ such that $F_{\gt}^B\equiv -1$, we will proceed as follows: first we will prove that 
	\begin{equation}\label{F g'' neg}
		\int_M F_{\tilde{g}''}^BdV_g<0,
	\end{equation}
	where $\tilde{g}''=(1+k^2\abs{\nabla\psi}^2)^{-\frac{1}{2}}g''$,
	for suitable choices of $\psi$ and $k$.
	Then, up to a perturbation of $k$, we will prove that
	$\abs{\bach_{g''}}_{g''}>0$ everywhere
	and the claim will follow by Lemma \ref{l-a2}. 
	To show \eqref{F g'' neg}, observe that
	$$
	g''= g' +d(k\psi)\otimes d(k\psi) = \psi \left[ g + d(2k\sqrt{\psi})\otimes d(2k\sqrt{\psi})\right] =:\psi \gb.
	$$
	Applying the same argument in the proof of \cite[Lemma 3.3]{Catino}, we deduce that
	\begin{align}\label{phi M}
		\Phi_M:=\int_M&F_{\tilde{g}''}^BdV_{\tilde{g}''}\\
		=\int_M &\left(S_{g'}+t|\bach_{g'}+E_{g'}(k\psi)|^{\frac{1}{2}}_{k\psi}\right)\,dV_{g'} -\int_M \frac{R'_{ij}\nabla_{g'}^i\psi\nabla_{g'}^j\psi}{1/k^2+|\nabla_{g'} \psi|_{g'}^2}\,dV_{g'}\notag\\
		& +\frac{3}{2}\int_M \left[\frac{\nabla^{g'}_{ip}\psi\nabla_{g'}^p\psi\nabla^{g'}_{iq}\psi\nabla_{g'}^q\psi}{(1/k^2+|\nabla_{g'}\psi|_{g'}^2)^2} -\frac{|\nabla^{g'}_{ij}\psi\nabla_{g'}^i\psi\nabla_{g'}^j\psi|^2}{(1/k^2+|\nabla_{g'}\psi|_{g'}^2)^3} \right]\,dV_{g'},\notag
	\end{align}
	where $\abs{\cdot}_{k\psi}$ denotes the norm with respect to $g'+d(k\psi)\otimes d(k\psi)$. \\
	
	With respect to the metric $g$, by standard formulas for conformal transformations (see \cite{CMbook}), we have
	\begin{align}\label{e-confinv}\nonumber
		S_{g'}&=\frac{1}{\psi}\left(S_g-3\frac{\Delta \psi}{\psi}+\frac{3}{2}\frac{|\nabla \psi|^2}{\psi^2}\right),\\\nonumber
		R'_{ij}&=R_{ij}-\frac{\psi_{ij}}{\psi}+\frac{3}{2}\frac{\psi_i\psi_j}{\psi^2}-\frac{1}{2}\frac{\Delta \psi}{\psi}g_{ij},\\
		B'_{ij}&=\frac{1}{\psi} B_{ij},\\\nonumber
		dV_{g'}&=\psi^{2}\,dV_g,\\\nonumber
		\nabla^{g'}_{ij} \psi &= \psi_{ij} -\frac{1}{\psi}\left(\psi_i\psi_j-\frac12|\nabla\psi|^2 g_{ij}\right).  
	\end{align}
	where $R'_{ij},\,B'_{ij}$ and $R_{ij},\,B_{ij}$ are relative to the metrics $g'$ and $g$, respectively.
	On the other hand, observe that
	$$
	g''= g' +d(k\psi)\otimes d(k\psi) = \psi \left[ g + d(2k\sqrt{\psi})\otimes d(2k\sqrt{\psi})\right] =:\psi \gb;
	$$
	moreover, in dimension four the Bach tensor is conformally covariant, namely
	$$
	B'_{ij}=\frac{1}{\psi}B_{ij}.
	$$
	
	Thus,
	\begin{align*}
		B'_{ij}+E'(k\psi)_{ij}&=B_{ij}''=\frac{1}{\psi}\ol{B}_{ij}\\
		&=\frac{1}{\psi}\pa{B_{ij}+E(2k\sqrt{\psi})_{ij}}\\
		&=B_{ij}+\frac{1}{\psi}E(2k\sqrt{\psi})_{ij}
	\end{align*}
	which implies that the "error term" of the Bach tensor satisfies:
	$$
	E_{g'}(k\psi)=\frac{1}{\psi}E_g(2k\sqrt{\psi}).
	$$
	In particular, the following relation is satisfied
	\begin{align*}
		\abs{\bach_{g''}}_{k\psi}^{\frac{1}{2}}=\abs{\bach_{g'}+E_{g'}(k\psi)}^{\frac{1}{2}}_{k\psi}=\frac{1}{\psi^2}\abs{\bach_g+E_g(2k\sqrt{\psi})}^{\frac{1}{2}}_{\gb}.
	\end{align*}
	Then, following the computations in \cite{Aubin2} and \cite{Catino} we get
	\begin{align}\label{e-int}\nonumber
		\Phi_M=&\int_M \left(S_{g}+\frac{t}{\psi}|\bach_g+E_g(2k\sqrt{\psi})|^{\frac{1}{2}}_{\gb}-\frac{R_{ij}\psi_i\psi_j}{\psi/k^2+|\nabla\psi|^2}\right)\psi\,dV_{g} \\\nonumber
		&+\int_M \frac{\psi_{ij}\psi^i\psi^j}{\psi/k^2+|\nabla\psi|^2}\,dV_g\\
		&+\frac{1}{k^2}\frac{3}{2}\int_M \frac{|\nabla\psi|^2}{\psi/k^2+|\nabla\psi|^2}\,dV_g-\frac{1}{k^2}\int_M \frac{\psi\Delta \psi}{\psi/k^2+|\nabla\psi|^2}\,dV_g\\\nonumber
		&+\frac{3}{2}\int_M \left[\frac{\psi_{ip}\psi^p\psi_{iq}\psi^q}{(\psi/k^2+|\nabla\psi|^2)^2} -\frac{|\psi_{ij}\psi^i\psi^j|^2}{(\psi/k^2+|\nabla\psi|^2)^3} \right]\psi\,dV_{g}\\\nonumber
		&+\frac{1}{k^2}\frac{3}{2}\int_M \frac{\tfrac14 |\nabla\psi|^6-|\nabla\psi|^2(\psi_{ij}\psi^i\psi^j)\psi}{(\psi/k^2+|\nabla\psi|^2)^3}\psi\,dV_g.
	\end{align}
	\subsection*{Step 2.}
	Let $p\in M$ and consider a local, normal, geodesic polar coordinate
	system $\rho,\omega_1,...,\omega_{n-1}$ defined in an open neighborhood $V$ 
	of $p$, in order to have
	\[
	g_{\rho\rho}=1, \quad 
	g_{\rho i}=0, \quad g_{ij}=\delta_{ij}+\rho^2 a_{ij}, \quad
	g^{\rho\rho}=1
	\]
	at $p$, where the index $i$ corresponds to the coordinate
	$\omega_i$, for $i=1,...,n-1$ and the coefficients
	$a_{ij}$ are of order 1;
	from now on, we use the
	index convention
	\[
	\alpha,\beta,\gamma,...=1,2,3,\rho, \quad
	i,j,k,...=1,2,3.
	\] The Christoffel symbols of the 
	Levi-Civita connection are written as
	\begin{equation}\label{chr}
	\Gamma_{\rho\rho}^{\rho}=\Gamma_{\rho i}^{\rho}=0, \quad 
	\Gamma_{ij}^{\rho}=-\dfrac{\rho}{2}\pa{a_{ij}+\rho
		\partial_\rho a_{ij}}.
	\end{equation}
	Let $B_r$ be the geodesic ball centered at $p$ of radius $0<r<r_0$,
	with $r_0$ such that $B_{r_0}\subset V$ and let $y=y(x)$ is a real $C^4$ function such that
	\[
	\begin{cases}
		&y(-x)=y(x), \mbox{ } \forall x\in\RR\\
		&y(x)=1, \mbox{ } \forall\abs{x}\geq 1\\
		&y(x)\geq\delta>0, \mbox{ } \forall x\in\RR\\
		&y'(x)>0, \mbox{ } \forall 0<x<1\\
		&y'(x)\geq 1, \mbox{ } \forall (1/4)^{1/(n-1)}\leq x\leq
		(3/4)^{1/(n-1)}\\
		&\abs{y''}\leq\frac{y'}{(1-x)} \mbox{ } \text{ as }x\ra1
	\end{cases}.
	\]
	Let $B_r=B_r(p)$ be the geodesic ball centered at $p$ of radius $0<r<r_0$, with $r_0$ such that $B_r\subset M$. For $p'\in B_r$, we choose
	$$
	\psi(p'):=y\left(\frac{\rho}{r}\right),\quad \rho=\text{dist}_g(p',p).
	$$
	From now on, to simplify the expressions, we will omit arguments in the functions: it will be clear that if $\psi$, $\psi_\rho$, etc. are computed at $p'\in B_r$, then $y, y', y''$ will be computed at $\rho/r$ with $\rho=\text{dist}_g(p',p)$. Moreover, we will denote by $C=C(n,\delta, t, p)>0$ some universal positive constant independent of $r$ and $k$. 
	
	\subsection*{Step 3}
	Arguing as in \cite{Catino}, it is possible to obtain an estimate for the terms not involving the Bach tensor, when they are restricted to the ball $B_r$. In particular, applying the same argument of Step 3 of \cite{Catino} we obtain that
	\begin{align}\label{e-ie1}
		\Phi_{B_r} \leq & \,t\int_{B_r} |\bach_g+E_g(2k\sqrt{\psi})|^{\frac{1}{2}}_{\gb}\,dV_{g} +C|B_r|+\frac{1}{r^2}\int_{B_r} y''\,dV_g+\frac{1}{k^2}\Theta,
	\end{align}
	where $\Phi_{B_r}$ denotes the quantity defined in \eqref{e-int} restricted to $B_r$. Note that this intermediate estimate, when $t=0$, coincides with the one of Aubin in \cite{Aubin2}.

	\subsection*{Step 4}
	We now give an estimate of the remaining terms, in which the Bach tensor appears.
	Since
	$$
	\gb=g + d(2k\sqrt{\psi})\otimes d(2k\sqrt{\psi}),
	$$
	for the sake of simplicity, we introduce
	\begin{equation}
		\eta:=2\sqrt{\psi},
	\end{equation}
	where 
	\[
	\psi(p'):=y\pa{\dfrac{\rho}{r}}
	\]
	and we have
	\begin{equation}\label{gbar}
	\gb=g + k^2d\eta\otimes d\eta.
	\end{equation}
	From \eqref{volume form and inverse aub def}, we have
	$$
	\gb^{\rho\rho}=\frac{1}{1+k^2 \eta_{\rho}^2},\quad \gb^{\rho i} =0, \quad \gb^{ij}=g^{ij}.
	$$
	In particular, since $\gb\geq g$, for every $(0,p)$-tensor $\operatorname{T}$ we immediately get that
	\begin{equation}\label{normgbarg}
		|\operatorname{T}_g|_{\gb} \leq |\operatorname{T}_g|_g\leq C,
	\end{equation}
	where $C$ is a constant.
	Note that, by definition of $\eta$ and \eqref{chr},
	\begin{align*}
		\eta_i&=0, \quad 
		\eta_\rho=\dfrac{1}{2r}\eta^{-1}y'\pa{\dfrac{\rho}{r}};\\
		\eta_{\rho\rho}&=-
		\dfrac{1}{4}\sq{\dfrac{1}{r}y'\pa{\dfrac{\rho}{r}}}^2
		\eta^{-3}+\dfrac{1}{2r^2}\eta^{-1}y''\pa{\dfrac{\rho}{r}};\notag\\
		\eta_{\rho i}&=0, \quad 
		\eta_{ij}=\dfrac{1}{4r}\eta^{-1}\rho\pa{a_{ij}+\rho
			\partial_\rho a_{ij}}y'\pa{\dfrac{\rho}{r}},
	\end{align*}
	which imply that there exists $C\in\erre$ such that
	\begin{align}\label{stime eta pt1}
		\frac{C^{-1} y'}{r}\leq\eta_\rho	\leq\frac{C y'}{r};\qquad\abs{\eta_{ij}}\leq Cr\abs{\eta_{\rho}}\leq C
	\end{align}
	and, more in general, if $r$ is sufficiently small,
	\[
	\eta_{\alpha}\leq \dfrac{C}{r}, \quad
	\abs{\eta_{\alpha\beta}}\leq\dfrac{C}{r^2}.
	\]
	Observe that, since $y$ only depends on $\rho$, 
	differentiating $\eta$ with respect to an angular
	coordinate does not raise the exponent of 
	$r$ at the denominator, while differentiating with
	respect to $\rho$ produces an additional $1/r$ 
	in the derivative: hence, 
	one can easily note that
	\begin{equation*} 
		\abs{\partial_{\alpha_1,...,\alpha_N}\eta}\leq\dfrac{C}{r^M}, 
	\end{equation*}
	where $M=\#\{i=1,...,N: \alpha_i=\rho\}$. In particular, we have 
	\begin{align}\label{estimeta}
		\abs{\eta_{\rho\rho}}\leq\frac{C}{r^2};\quad\abs{\eta_{\rho\rho\rho}}\leq\frac{C}{r^3};\quad\abs{\eta_{\rho\rho\rho\rho}}\leq\frac{C}{r^4};\quad
		\abs{\eta_{ij}}\leq C; \quad\abs{\partial_{\rho}\eta_{ij}}\leq \frac{C}{r};\quad \abs{\partial_{t}\partial_{\rho}\eta_{ij}}\leq \frac{C}{r}; \quad \abs{\partial_{i_1...i_l}\eta_{ij}}\leq C.
	\end{align}
	Moreover, by assumption $\abs{y''(x)}\leq y'/(1-\frac{\rho}{r})$ and the definition of $y$, we exploit 
	\begin{align}\label{eta rho rho e eta rho}
		\abs{\eta_{\rho\rho}}\leq C\pa{\frac{(y')^2}{r^2}+\frac{\abs{y''}}{r^2}}\leq\frac{C\abs{y'}}{r^2(\frac{\rho}{r}-1)}\leq\frac{C\abs{\eta_{\rho}}}{r(\frac{\rho}{r}-1)}\leq\frac{C\abs{\eta_{\rho}}}{r-\rho}
	\end{align}
	and
	\begin{align}\label{eta ang ed eta rho}
		\begin{cases}
			\abs{\partial_{\rho}\eta_{ij}}=\abs{\partial_\rho\sq{\dfrac{1}{4r}\eta^{-1}\rho\pa{a_{ij}+\rho
						\partial_\rho a_{ij}}y'\pa{\dfrac{\rho}{r}}}}\leq \frac{C\abs{\eta_{\rho}}}{1-\frac{\rho}{r}}\leq\frac{C\abs{\eta_{\rho}}}{r-\rho};\\
			\abs{\partial_{t}\partial_{\rho}\eta_{ij}}=\abs{\partial_{t}\partial_\rho\sq{\dfrac{1}{4r}\eta^{-1}\rho\pa{a_{ij}+\rho
						\partial_\rho a_{ij}}y'\pa{\dfrac{\rho}{r}}}}\leq \frac{C\abs{\eta_{\rho}}}{1-\frac{\rho}{r}}\leq\frac{C\abs{\eta_{\rho}}}{r-\rho};
		\end{cases}
	\end{align}
	\noindent
	summarizing, we have
	\begin{align}\label{stime eta}
		\abs{\partial_\rho\eta_{p}}&=\abs{\eta_{\rho\rho}+\Gamma^\alpha_{\rho\rho}\eta_{\alpha}}=\abs{\eta_{\rho\rho}}\leq\frac{C\abs{\eta_{\rho}}}{r-\rho};\notag\\
		\abs{\partial_\rho\eta_{\rho\rho}}&=\abs{\eta_{\rho\rho\rho}+2\Gamma^\alpha_{\rho\rho}\eta_{\alpha\rho}}=\abs{\eta_{\rho\rho\rho}}\leq\frac{C}{r^3}\pa{\abs{(y')^3}+\abs{y'y''}+\abs{y'''}}\leq\frac{C}{r^3};\notag\\
		\abs{\partial_\rho\eta_{ij}}&\leq\frac{C\abs{\eta_{\rho}}}{1-\frac{\rho}{r}}=\frac{Cr\abs{\eta_{\rho}}}{r-\rho};\notag\\
		\abs{\partial_{\rho}\partial_{\rho}\eta_{ij}}&=\abs{\frac{\eta_{\rho\rho}\rho y'L}{r}+\frac{\eta_\rho^2\rho y'L}{r}+\frac{\eta_{\rho}y'L}{r}+\frac{\eta_\rho\rho y''L}{r^2}+\frac{\eta_{\rho}y'L}{r}+\frac{Ly''}{r^2}+\frac{\eta_{\rho}\rho y''L}{r}+\frac{\rho y'''L}{r^3}}\\&\leq\frac{C\abs{y'''}}{r^2}\leq\frac{C}{r^2};\notag\\
	\abs{\partial_l\eta_{ij}}&=\abs{\dfrac{1}{4r}\eta^{-1}\rho\pa{\partial_l a_{ij}+\rho
			\partial_l\partial_\rho a_{ij}}y'\pa{\dfrac{\rho}{r}}}\leq Cr\abs{\eta_{\rho}}\leq C;\notag\\
		\abs{\partial_t\partial_l\eta_{ij}}&=\abs{\dfrac{1}{4r}\eta^{-1}\rho\pa{\partial_t\partial_l a_{ij}+\rho
			\partial_t\partial_l\partial_\rho a_{ij}}y'\pa{\dfrac{\rho}{r}}}\leq Cr\eta_{\rho}\leq C;\notag\\
		\abs{\partial_t\partial_\rho\eta_{ij}}&\leq C\abs{\partial_\rho\eta_{ij}}\leq\frac{rC\abs{\eta_{\rho}}}{r-\rho},\notag
	\end{align}
	where $L$ denotes a bounded quantity and $C$ is a constant. 	Depending on the term we need to estimate, we are going to use \eqref{estimeta} or \eqref{eta rho rho e eta rho} and \eqref{eta ang ed eta rho}.\\
	To give an estimate of the remaining integral depending on the Bach tensor we exploit the following
	\begin{lemma}\label{l-resto B}
		We have
	\begin{align}\label{estB}
		\int_{B_r}\abs{\bach_g+E_g(2k\sqrt{\psi})}_{\gb}^{\frac{1}{2}}dV_g=\int_{B_r}\abs{\bach_{\gb}}^{\frac{1}{2}}_{\gb}dV_g\leq C\abs{B_r}+\frac{C}{k}\Theta,
	\end{align}
		for some constant $C=C(\delta,t,p)$ and a continuous function $\Theta=\Theta(p,\frac{1}{k},r)>0$; here $\abs{B_r}$ denotes the volume of $B_r$.
	\end{lemma}
	\begin{proof}
		Given a tensor $\operatorname{T}$ in the metric $g$, we will denote as $\ol{\operatorname{T}}$ the same tensor with respect to the metric $\gb$.\\
		We recall that 
		by \eqref{resto bach}, we have
		\begin{align}\label{barbach}
			\abs{\ol{\bach}}_{\gb}&=\abs{\frac{1}{2}\sq{\bar{\Delta} \ol{\ric}-\frac{1}{3}\ol{\hess}(\ol{S})+2\ol{\ric}*\ol{\riem}
					-\frac{2}{3}\ol{S}\,\ol{\ric}-\frac{1}{6}\bar{\Delta} \ol{S} \ol{g}-\frac{1}{2}\pa{\abs{\ol{\ric}}_{\ol{g}}^2-\frac{\ol{S}^2}{3}}\ol{g}}}_{\gb}\\
				&\leq\frac{1}{2}\sq{\abs{\bar{\Delta} \ol{\ric}}_{\gb}+\frac{1}{3}\abs{\ol{\hess}(\ol{S})}_{\gb}+2\abs{\ol{\ric}*\ol{\riem}}_{\gb}+\frac{2}{3}\abs{\ol{S}\,\ol{\ric}}_{\gb}+\frac{1}{6}\abs{\bar{\Delta} \ol{S} \ol{g}}_{\gb}+\frac{1}{2}\abs{\pa{\abs{\ol{\ric}}_{\ol{g}}^2-\frac{\ol{S}^2}{3}}\ol{g}}_{\gb}},  \notag
		\end{align}
		where $*$ denotes the contraction $R_{\gamma\delta}R_{\alpha\gamma\beta\delta}$.
		The computations in Section \ref{app} show that
		\begin{equation*}
			\begin{cases}
				\abs{\bar{\Delta}\ol{\ric}}_{\gb}\leq C\pa{1+\frac{1}{r-\rho}+\frac{1}{k^2}\Theta};\\
				\abs{\ol{\hess}(\ol{S})}_{\gb}\leq C\pa{1+\frac{1}{r-\rho}+\frac{1}{k^2}\Theta};\\
				\abs{\ol{\ric}*\ol{\riem}}_{\gb}\leq C\pa{1+\frac{1}{k^2}\Theta};\\
				\abs{\ol{S}\,\ol{\ric}}_{\gb}\leq 	C\pa{1+\frac{1}{k^2}\Theta};\\
				\abs{\bar{\Delta} \ol{S} \ol{g}}_{\gb}\leq C\pa{1+\frac{1}{r-\rho}+\frac{1}{k^2}\Theta};\\
				\abs{\pa{\abs{\ol{\ric}}_{\gb}-\frac{\ol{S}^2}{3}}\gb}_{\gb}\leq 	C\pa{1+\frac{1}{k^2}\Theta},
			\end{cases}
		\end{equation*}
		for some $C=C(n,\delta,t,p)>0$ and $\Theta=\Theta(p,1/k,r)>0$; hence we have
		$$\abs{\ol{\bach}}_{\gb}\leq C\pa{1+\frac{1}{r-\rho}+\frac{1}{k^2}\Theta}.$$
		As a consequence, we get
		\begin{equation}\label{estbach}
		\abs{\ol{\bach}}_{\gb}^{\frac{1}{2}}\leq C+\frac{C}{(r-\rho)^{\frac{1}{2}}}+\frac{C}{k}\Theta.
		\end{equation}
		Note that, when $\rho\ra r$, we have that $1/(r-\rho)^{\frac{1}{2}}$ is integrable and thus \eqref{estbach} implies 
		\[t\int_{B_r}\abs{\ol{\bach}}_{\gb}^{\frac{1}{2}}\leq C\abs{B_{r}}+\frac{C}{k}\Theta.\]	
	\end{proof}
	\begin{rem}
	We point out that, in \cite[Lemma 4.1]{Catino}, there is a misprint: the power of $k$ should be
	$-1$ instead of $-2$, which, however, does not affect the validity of the arguments. 
	\end{rem}
	\subsection*{Step 5} 
	Using Lemma \ref{l-resto B} in \eqref{e-ie1}, we obtain
	\begin{equation}\label{e-ie3}
		\Phi_{B_r} \leq  \,C|B_r|+\frac{1}{r^2}\int_{B_r} y''\,dV_g+\frac{1}{k}\Theta.
	\end{equation}
	Now, we have to make sure that $\abs{\bach_{\gb}}_{\gb}>0$
	on $M$: in order to do so, assume that 
	there exists a point $p'\in B_r$ such that 
	$\abs{\bach_{\gb}}_{\gb}=0$ vanishes at $p'$.
	If we evaluate the components of 
	$\bach_{\gb}$ at $p'$, we obtain that the equation
	$\ol{B}_{\alpha\beta}(p')=0$ is a polynomial equation of
	finite degree $N_{\alpha,\beta}$
	in the variable $k$, for $\alpha,\beta=1,2,3,\rho$,
	up to
	multiplying the equation for $(1+k^2\eta_\rho^2)^{\lambda}$, where
	$\lambda=\lambda(\alpha,\beta)$ is the highest
	power of $(1+k^2\eta_\rho^2)$ appearing in the
	denominators of the expression of $\ol{B}_{\alpha\beta}$
	(this can be done since all the denominators in these
	expressions are of the
	form $(1+k^2\eta_\rho)^{\gamma}$, as can be seen 
	in the previous computations and in Section \ref{app}:
	this means that $\abs{\bach_{\gb}}_{\gb}=0$ at $p'$ if and only
	if $k$ is a root of all the polynomials, which
	implies that
	\[
	k\in A:=\set{k_1,...,k_L}, \quad L\leq 
	N:=\min_{\alpha,\beta=1,2,3,\rho}N_{\alpha,\beta}
	\]
	where $k_1,...,k_L$ are the common roots of the
	polynomials. 
	We observe that, since $k$ is a fixed constant in
	\eqref{gbar}, the roots of the polynomials 
	$\ol{B}_{\alpha\beta}(q)$ have to be
	contained in $A$, for every $q\in B_r$
	such that $\abs{\bach_{\gb}}_{\gb}=0$ at $q$.
	Therefore, in order to have that $\bach_{\gb}$ does not
	vanish on $B_r$, it is sufficient to choose 
	$k$ outside of $A$ in \eqref{gbar}: hence,
	we can conclude that $\abs{\bach_{\gb}}_{\gb}\neq 0$
	on $B_r$.
	
	In order to conclude the proof, we apply the same argument as the one in Step 5 of \cite{Catino}:
	for the sake of completeness, we include it here. 
	Using \eqref{e-ie3} and that, by assumption, $y'(x)\geq 1$ for all $(1/4)^{1/(n-1)}\leq x\leq (3/4)^{1/(n-1)}$, we obtain
	\begin{align*}
		\Phi_{B_r}  &\leq  \,C\left(1+\frac{1}{r}\right)|B_r|-\frac{n-1}{r}|\SS^{n-1}|\inf_M \sqrt{\text{det}g_{ij}} \int_{r(\tfrac{1}{4})^{1/(n-1)}}^{r(\tfrac{3}{4})^{1/(n-1)}}\rho^{n-2}\,d\rho+\frac{1}{k}\Theta\\
		&\leq C\left(1+\frac{1}{r}\right)|B_r|-\frac{C_2}{r^2}|B_r|+\frac{1}{k}\Theta,
	\end{align*}
	where we used the fact that $|B_r|\sim c r^n$ as $r\to 0$. In particular, there exist a continuous function $\lambda(p)> 0$ and, for $p\in M$ fixed, a continuous function $\Theta_p(r)>0$ in $r$, for $0<r<r_0$, such that
	$$
	\Theta(p, 1/k, t)\leq \Theta_p(r),
	$$
	and
	\begin{equation}\label{e-ie4}
		\Phi_{B_r}  \leq \left[C\left(1+\frac{1}{r}\right)-\frac{\lambda}{r^2}\right]|B_r| +\frac{1}{k}\Theta_p(r).
	\end{equation}
	Since, by \eqref{fpositive}, $F_g^B=S_g+t|\bach_g|_g^{\frac{1}{2}}\geq 0$, given $\nu>0$, there exists a positive radius $0<r_1<r_0$ such that
	\begin{equation}\label{e-r1}
		\frac{\lambda}{r_1^2}-C\left(1+\frac{1}{r_1}\right)-1\geq\nu \widehat{F}_g^B,
	\end{equation}
	where $\widehat{F}_g^B:=\left(\int_M F_g^B\,dV_g\right)/\text{Vol}_g(M)$.
	
	We consider $h$ disjoint geodesic balls 
	$B^j_{r_1}(p_j)$ of radius $r=r_1$, $j=1,...,h$, together
	with the corresponding function $\psi^{[j]}$, as
	constructed above, such that, for 
	$\nu$ sufficiently large,
	\[
	\sum_{j=1}^h\abs{B^j_{r_1}(p_j)}>\dfrac{1}{\nu}\mathrm{Vol}_g(M);
	\] 
	we define $A_j$ as before and, on $B^j$, we choose
	$k_j$ such that the Bach tensor of the deformed metric
	does not vanish on $B^j$ and $k_j\not\in A_i$ for every 
	$i=1,...,h$. 
	
	On every $B^j$, we set
	\[
	k\geq\max\left\{1,\sup_{j=1,\ldots,h} 
	\frac{\Theta_{p_j}(r_1)}{|B^j_{r_1}(p_j)|}\right\}, \quad
	k\not\in\bigcup_{j=1}^hA_j,
	\]
	which is possible since $\bigcup_{j=1}^hA_j$ is a finite set.
	From \eqref{e-ie4} and \eqref{e-r1}, for all $j=1,\ldots,h$, we get
	$$
	\Phi_{B^j_{r_1}}\leq -\nu \widehat{F}_g^B |B^j_{r_1}(p_j)| -|B^j_{r_1}(p_j)| +
	\frac{1}{k}\Theta_{p_j}(r_1) \leq -\nu \widehat{F}_g^B |B^j_{r_1}(p_j)|.
	$$
	If we define
	\[
	\psi:=\begin{cases}
		1, \quad &\text{on } M\setminus\bigcup_{j=1}^h B^j\\
		\psi^{[j]}, \quad &\text{on } B^j,
	\end{cases}
	\]
	we obtain 
	$$
	\Phi_M\leq \int_M F_g^B\,dV_g -\nu \widehat{F}_g^B \sum_{j=1}^{h}|B^j_{r_1}(p_j)| =
	\widehat{F}_g^B\left(\text{Vol}_g(M)-\nu\sum_{j=1}^{h}|B^j_{r_1}(p_j)|\right)< 0;
	$$
	furthermore, by our choice of $k$ and the fact that $\abs{\bach_g}_g>0$ on $M$, we
	obtain that $\abs{\bach_{\gb}}_{\gb}>0$ on $M$.
	By Lemma \ref{l-a2}, there exists a metric $\gt\in[\gb]$ such that $F^B_{\gt}\equiv -1$.
	
\
	
	Finally, we consider the case $t\leq0$; by \cite{Aubin1, Aubin2} we know that, on a closed $4$-dimensional manifold, there exists a Riemannian metric $g'$ with constant scalar curvature $-1$, which is constructed \emph{via} the same
	deformation we exploited in the previous case, starting
	from a reference metric. Hence,
	let $g$ be a Riemannian metric on $M$ 
	such that $\abs{\bach_g}_g>0$ at every point of $M$: 
	exploiting Aubin's proof and the previous argument on the choice of $k$, we 
	can produce a metric $\gb$ such that $\int_M S_{\gb}dV_{\gb}<0$ and $\abs{\bach_{\gb}}_{\gb}>0$ on
	$M$. Therefore, since  $t\leq 0$, obviously $\int_M F^B_{\gb}dV_{\gb}<0$ and, by Lemma \ref{l-g2}, there exists a metric $\gt\in[\gb]$ such that $F^B_{\gt}\equiv -1$.
	
\

This concludes the proof of Theorem \ref{t-main}.\begin{flushright}$\square$\end{flushright}
	
\

	\section{Estimates on the deformed Bach tensor}\label{app}
	In this section we prove the estimate \eqref{estbach} that we used in the fourth step of the proof of Theorem \ref{t-main}:
	for the sake of simplicity, we will write $\operatorname{T}=\operatorname{T}_g$ and 
	$\ol{\operatorname{T}}=\operatorname{T}_{\gb}$ for every considered tensor $\operatorname{T}$. 
	We recall that, for the metric $\gb$ defined in \eqref{gbar}, we have
		\begin{equation}\label{normbarbach}
			\abs{\ol{\bach}}_{\gb}=\abs{\frac{1}{2}\sq{\bar{\Delta} \ol{\ric}-\frac{1}{3}\ol{\hess}(\ol{S})+2\ol{\ric}*\ol{\riem}
					-\frac{2}{3}\ol{S}\,\ol{\ric}-\frac{1}{6}\bar{\Delta} \ol{S} \ol{g}-\frac{1}{2}\pa{\abs{\ol{\ric}}_{\ol{g}}^2-\frac{\ol{S}^2}{3}}\ol{g}}}_{\gb}.
		\end{equation}
To prove the validity of \eqref{estbach}, we analyze the components of $\ol{B}$ separately, proving that each term in \eqref{normbarbach} satisfies
		\begin{align}\label{ineq to prove}
			\abs{\,\cdot\,}_{\gb}\leq C\pa{1+\frac{1}{r-\rho}+\frac{1}{k}\Theta}\quad\text{in }B_r,
		\end{align} 
		where $C$ is a positive constant and $\Theta=\Theta(p,\frac{1}{k},r)$ is a positive continuous function.
		For the sake of simplicity, given a $(0,q)$-tensor $\operatorname{T}$, we will denote
		\begin{align*}
			\abs{T_{\alpha_1...\alpha_q}}_{\ol{g}}^2=T_{\alpha_1...\alpha_q}T_{\beta_1...\beta_q}\ol{g}^{\alpha_1\beta_1}...\,\ol{g}^{\alpha_q\beta_q};
		\end{align*}
		for instance, on a $(0,2)$-tensor we have
		\begin{align*}
			\abs{T_{ij}}_{\gb}^2&= T_{ij}T_{lt}\gb^{il}\gb^{jt}=\abs{T_{ij}}_g^2\leq C\abs{T_{ij}}^2;\\
			\abs{T_{i\rho}}_{\gb}^2&= T_{i\rho}T_{j\rho}\gb^{ij}\gb^{\rho\rho}=\frac{1}{1+k^2\eta_{\rho}^2}\abs{T_{i\rho}}^2_g\leq \frac{C}{1+k^2\eta_{\rho}^2}\abs{T_{i\rho}}^2;\\
			\abs{T_{\rho\rho}}_{\gb}^2&= T_{\rho\rho}T_{\rho\rho}\gb^{\rho\rho}\gb^{\rho\rho}=\frac{1}{(1+k^2\eta_{\rho}^2)^2}\abs{T_{\rho\rho}}^2_g=\frac{1}{(1+k^2\eta_{\rho}^2)^2}\abs{T_{\rho\rho}}^2.
		\end{align*}
		Note that when we consider Aubin's deformation, we have (see Section \ref{sec-3})
		\begin{align*}
			\ol{R}_{\alpha\beta\gamma\delta}&=R_{\alpha\beta\gamma\delta}+E^R_{\alpha\beta\gamma\delta};\\
			\ol{R}_{\alpha\beta}&=R_{\alpha\beta}+F_{\alpha\beta};\\
			\ol{S}&=S+H;
		\end{align*}
		where 
		\begin{align*}
			E^R_{ijlt}=&\frac{k^2}{1+k^2\eta_{\rho}^2}(\eta_{il}\eta_{jt}-\eta_{it}\eta_{jl});\\
			E^R_{i\rho lt}=&0;\\
			E^R_{i\rho l\rho}=&\frac{k^2\eta_{\rho\rho}\eta_{il}}{1+k^2\eta_{\rho}^2};\\
			F_{ij}=&-\frac{k^2\eta_{\rho}^2R_{i\rho j\rho}}{1+k^2\eta_{\rho}^2}+\frac{Ck^2(\eta_{ANG}\eta_{ANG})_{ij}}{1+k^2\eta_{\rho}^2}+\frac{k^2\eta_{\rho\rho}^2\eta_{ij}}{(1+k^2\eta_{\rho}^2)^2};\\
			F_{i\rho}=&0;\\
			F_{\rho\rho}=&\frac{k^2\eta^i_{i}\eta_{\rho\rho}}{1+k^2\eta_{\rho}^2};\\
			H=&-\frac{2k^2}{1+k^2\eta_{\rho}^2}R_{\rho\rho}\eta_{\rho}^2+\frac{k^2}{1+k^2\eta_{\rho}^2}\pa{(\Delta\eta)^2-\abs{\hess (\eta)}^2}-\frac{2k^4}{(1+k^2\eta_{\rho}^2)^2}\pa{\Delta \eta\eta_{\rho}^2\eta_{\rho\rho}-\eta_{\rho}^2\eta_{\rho\rho}^2}\\
			=&L+\frac{L}{1+k^2\eta_{\rho}^2}+\frac{C k^2\eta_{ANG}\eta_{ANG}}{1+k^2\eta_{\rho}^2}+\frac{2k^2\eta_{\rho\rho}\eta_{ll}}{(1+k^2\eta_{\rho}^2)^2}.
		\end{align*}
		We use the notation $\eta_{ANG}$ to denote terms of the type $\eta_{ij}$, where $i,j$ are angular coordinates (with this notation, we have coupled the angular terms $\eta_{ll}\eta_{ij}$ and $\eta_{il}\eta^l_j$ in $F_{ij}$ and the terms $\eta_{ll}^2$ and $\eta_{ij}\eta_{ij}$ in $H$, since for small $r$, they have the same behavior).	Then, by \eqref{estimeta} and \eqref{stime eta pt1} we get:
		\begin{align}\label{resti bar mod}
			\abs{E^R_{ijlt}}=&\abs{\frac{k^2}{1+k^2\eta_{\rho}^2}(\eta_{il}\eta_{jt}-\eta_{it}\eta_{jl})}\leq\abs{\frac{k^2Lr^2\eta_\rho^2}{k^2\eta_{\rho}^2}}\leq Cr^2\leq C;\\
			\abs{E^R_{i\rho lt}}=&0;\notag\\
			\abs{E^R_{i\rho l\rho}}=&\abs{\frac{k^2\eta_{\rho\rho}\eta_{il}}{1+k^2\eta_{\rho}^2}}\leq\abs{\frac{k^2L}{k^2\eta_{\rho}^2}\frac{r\eta_{\rho}^2}{r^2}}\leq\frac{C}{r}\leq C\Theta;\notag\\
			\abs{F_{ij}}=&\abs{-\frac{k^2\eta_{\rho}^2R_{i\rho j\rho}}{1+k^2\eta_{\rho}^2}+\frac{Ck^2(\eta_{ANG}\eta_{ANG})_{ij}}{1+k^2\eta_{\rho}^2}+\frac{k^2\eta_{\rho\rho}^2\eta_{ij}}{(1+k^2\eta_{\rho}^2)^2}}\notag\\
			\leq& \abs{\frac{Lk^2\eta_{\rho}^2}{k^2\eta_{\rho}^2}}+\abs{\frac{Lk^2\eta_{\rho}^2r^2}{k^2\eta_{\rho}^2}}+\abs{\frac{k^2L}{r^2\frac{k^4}{r^4}\pa{\frac{r^2}{k^2}+\frac14\eta^{-2}(y')^2}^2}}\leq C+C+\frac{C}{k^2}\Theta
			&\leq C\pa{1+\frac{1}{k^2}\Theta};\notag\\
			\abs{F_{i\rho}}=&0;\notag\\
			\abs{F_{\rho\rho}}=&\abs{\frac{k^2\eta^i_{i}\eta_{\rho\rho}}{1+k^2\eta_{\rho}^2}}\leq C\Theta;\notag\\
			\abs{H}=&\abs{L+\frac{L}{1+k^2\eta_{\rho}^2}+\frac{C k^2\eta_{ANG}\eta_{ANG}}{1+k^2\eta_{\rho}^2}+\frac{2k^2\eta_{\rho\rho}\eta_{ll}}{(1+k^2\eta_{\rho}^2)^2}}\notag\\
			\leq& C+C+C+r^2\frac{C}{k^2}\Theta\leq C\pa{1+\frac{1}{k^2}\Theta}.\notag
		\end{align}
		Note that we have not used \eqref{eta rho rho e eta rho} and \eqref{eta ang ed eta rho}, however il will be useful to have an estimate of $F_{\rho\rho}$ in terms of $\frac{1}{r-\rho}$ (since we are going to use it later to compute some of the remainders of $\ol{\Delta}\ol{\ric}$):
		\begin{align}\label{F rho rho alternative}
			\abs{F_{\rho\rho}}\leq\abs{\frac{k^2\eta_\rho^2\frac{r}{r-\rho}}{r\pa{1+k^2\eta_{\rho}^2}}}\leq \frac{C}{r-\rho}.
		\end{align}
		It follows
		\begin{align}\label{resti bar}
			\abs{E^R_{ijlt}}_{\gb}=&\abs{E^R_{ijlt}}_g\leq C\abs{E^R_{ijlt}}\leq Cr^2\leq C;\\
			\abs{E^R_{i\rho lt}}_{\gb}=&0;\notag\\
			\abs{E^R_{i\rho l\rho}}_{\gb}=&\frac{1}{1+k^2\eta_{\rho}^2}\abs{\frac{k^2\eta_{\rho\rho}\eta_{il}}{1+k^2\eta^2_{\rho}}}_g\leq\frac{C}{1+k^2\eta_{\rho}^2}\abs{\frac{k^2\eta_{\rho\rho}\eta_{il}}{1+k^2\eta^2_{\rho}}}\leq\frac{k^2 C}{r^2\frac{k^4}{r^4}\pa{\frac{r^2}{k^2}+\frac{\eta^{-2}}{4}(y')^2}^2}\leq r^2\frac{C}{k^2}\Theta\leq\frac{C}{k^2}\Theta;\notag\\
			\abs{F_{ij}}_{\gb}=& \abs{F_{ij}}_g\leq C\abs{F_{ij}}\leq C\pa{1+\frac{1}{k^2}\Theta};\notag\\
			\abs{F_{i\rho}}_{\gb}=&0;\notag\\
			\abs{F_{\rho\rho}}_{\gb}=&\frac{1}{1+k^2\eta_{\rho}^2}\abs{\frac{k^2\eta^i_{i}\eta_{\rho\rho}}{1+k^2\eta_{\rho}^2}}_g=\frac{1}{1+k^2\eta_{\rho}^2}\abs{\frac{k^2\eta^i_{i}\eta_{\rho\rho}}{1+k^2\eta_{\rho}^2}}\leq \frac{C}{k^2}\Theta;\notag\\
			\abs{H}_{\gb}=&\abs{H}_g=\abs{H}\leq C\pa{1+\frac{1}{k^2}\Theta}.\notag
		\end{align}
		As a consequence, exploiting \eqref{resti bar}, a straightforward computation proves that the terms of \eqref{barbach} in which $\riem,\ric$ and $S$ appear satisfy \eqref{ineq to prove}. For instance, 
		\begin{align}\label{SRic}
			\abs{\ol{S}\,\ol{\ric}}_{\gb}=&\abs{SR_{\alpha\beta}+HR_{\alpha\beta}+SF_{\alpha\beta}+HF_{\alpha\beta}}_{\gb}\\
			=&\abs{SR_{ij}+HR_{ij}+SF_{ij}+HF_{ij}+SR_{i\rho}+HR_{i\rho}+SR_{\rho\rho}+HR_{\rho\rho}+SF_{\rho\rho}+HF_{\rho\rho}}_{\gb}\notag\\
			\leq&\abs{S}_{\gb}\abs{R_{ij}}_{\gb}+\abs{H}_{\gb}\abs{R_{ij}}_{\ol{g}}+\abs{S}_{\gb}\abs{F_{ij}}_{\gb}+\abs{H}_{\gb}\abs{F_{ij}}_{\gb}+\abs{S}_{\gb}\abs{R_{i\rho}}_{\gb}+\abs{H}_{\gb}\abs{R_{i\rho}}_{\gb}\notag\\
			&+\abs{S}_{\gb}\abs{R_{\rho\rho}}_{\gb}+\abs{H}_{\gb}\abs{R_{\rho\rho}}_{\gb}+\abs{S}_{\gb}\abs{F_{\rho\rho}}_{\gb}+\abs{H}_{\gb}\abs{F_{\rho\rho}}_{\gb}\notag\\
			\leq& C+C\pa{1+\frac{1}{k^2}\Theta}+C\pa{1+\frac{1}{k^2}\Theta}+C\pa{1+\frac{1}{k^2}\Theta}+C+C\pa{1+\frac{1}{k^2}\Theta}\notag\\
			&+C+C\pa{1+\frac{1}{k^2}\Theta}+\frac{C}{k^2}\Theta+\frac{C}{k^2}\Theta\notag\\
			\leq& C\pa{1+\frac{1}{k^2}\Theta},\notag
		\end{align}
		A similar computation shows that
		\begin{align*}
			\abs{\ol{\ric}}_{\gb}\leq C\pa{1+\frac{1}{k^2}\Theta};\qquad\abs{\ol{S}}_{\gb}\leq C\pa{1+\frac{1}{k^2}\Theta}
		\end{align*}
		and, in particular:
		\begin{align}\label{bach pezzi}
			\begin{cases}
				\abs{\ol{\ric}*\ol{\riem}}_{\gb}\leq C\pa{1+\frac{1}{k^2}\Theta};\\
				\abs{\pa{\abs{\ol{\ric}}_{\gb}-\frac{\ol{S}^2}{3}}\gb}_{\gb}\leq C\pa{1+\frac{1}{k^2}\Theta}.
			\end{cases}
		\end{align}
		Thus, it remains to analyze the covariant derivatives of $\ol{\ric}$ and $\ol{S}$. 	Let
		\begin{align}\label{def G}
			G^\alpha_{\beta\gamma}:=\ol{\Gamma}^{\alpha}_{\beta\gamma}-\Gamma^\alpha_{\beta\gamma}=\frac{k^2\eta^\alpha\eta_{\beta\gamma}}{1+k^2\eta_\rho^2},
		\end{align}
		then by \eqref{stime eta pt1}, \eqref{estimeta} and \eqref{def G} we deduce
		\begin{align}\label{G}
			G^i_{jk}&=G^i_{j\rho}=G^i_{\rho\rho}=0\notag;\\
			\abs{G^\rho_{ij}}&=\abs{\frac{k^2\eta_\rho\eta_{ij}}{1+k^2\eta_\rho^2}}\leq \frac{k^2Cr\abs{\eta_{\rho}}^2}{1+k^2\abs{\eta_{\rho}}^2}\leq r C;\\
			\abs{G^\rho_{\rho\rho}}&=
			\abs{\frac{k^2\eta_{\rho}\eta_{\rho\rho}}{1+k^2\eta_\rho^2}}\leq\frac{k^2C}{r^2\frac{k^2}{r^2}\pa{\frac{r^2}{k^2}-\frac{\eta^{-2}}{4}\abs{y'}^2}}\leq C\Theta,\notag
		\end{align}
		
		By definition, we need to compute
		\begin{align}\label{ricci lap}
			\bar{\Delta}\ol{\ric}&=\ol{g}^{\rho\rho}\ol{\nabla}_\rho\ol{\nabla}_{\rho}(\ol{R}_{\rho\rho}+\ol{R}_{i\rho}+\ol{R}_{ij})+\ol{g}^{lt}\ol{\nabla}_l\ol{\nabla}_t(\ol{R}_{\rho\rho}+\ol{R}_{i\rho}+\ol{R}_{ij})\\
			&=\frac{1}{1+k^2\eta_{\rho}^2}\ol{\nabla}_\rho\ol{\nabla}_{\rho}(\ol{R}_{\rho\rho}+\ol{R}_{i\rho}+\ol{R}_{ij})+g^{lt}\ol{\nabla}_l\ol{\nabla}_t(\ol{R}_{\rho\rho}+\ol{R}_{i\rho}+\ol{R}_{ij}),\notag
		\end{align}
		where
		\begin{align}\label{nabla ricci}
			\ol{\nabla}_{\gamma}\ol{R}_{\alpha\beta}=\nabla_{\gamma}R_{\alpha\beta}-G^\delta_{\gamma\alpha}R_{\delta\beta}-G^\delta_{\gamma\beta}R_{\alpha\delta}-G^\delta_{\gamma\alpha}F_{\delta\beta}-G^\delta_{\gamma\beta}F_{\alpha\delta}
		\end{align}
		and
		\begin{align}\label{nabla nabla ricci}
			\ol{\nabla}_\sigma\ol{\nabla}_\gamma\ol{R}_{\alpha\beta}=&\nabla_{\sigma}\nabla_{\gamma}R_{\alpha\beta}+\nabla_{\sigma}\nabla_{\gamma}F_{\alpha\beta}-G^\tau_{\sigma\alpha}(\nabla_\gamma R_{\tau\beta})-G^\tau_{\sigma\beta}(\nabla_\gamma R_{\alpha\tau})-G^\tau_{\sigma\gamma}(\nabla_\tau R_{\alpha\beta})\\
			&-G^\tau_{\sigma\alpha}(\nabla_\gamma F_{\tau\beta})-G^\tau_{\sigma\beta}(\nabla_\gamma F_{\alpha\tau})-G^\tau_{\sigma\gamma}(\nabla_\tau F_{\alpha\beta})\notag\\
			&-\partial_{\sigma}G^\delta_{\gamma\alpha}R_{\delta\beta}-\partial_{\sigma}G^\delta_{\gamma\beta}R_{\alpha\delta}-\partial_{\sigma}G^\delta_{\gamma\alpha}F_{\delta\beta}-\partial_{\sigma}G^\delta_{\gamma\beta}F_{\alpha\delta}\notag\\
			&-G^\delta_{\gamma\alpha}\partial_{\sigma}R_{\delta\beta}-G^\delta_{\gamma\beta}\partial_{\sigma}R_{\alpha\delta}-G^\delta_{\gamma\alpha}\partial_{\sigma}F_{\delta\beta}-G^\delta_{\gamma\beta}\partial_{\sigma}F_{\alpha\delta}+\Gamma^\tau_{\sigma\gamma}(G^\delta_{\tau\alpha}R_{\delta\beta})\notag\\
			&+\Gamma^\tau_{\sigma\alpha}(G^\delta_{\gamma\tau}R_{\delta\beta})+\Gamma^\tau_{\sigma\beta}(G^\delta_{\gamma\alpha}R_{\delta\tau})+\Gamma^\tau_{\sigma\gamma}(G^\delta_{\tau\beta}R_{\alpha\delta})+\Gamma^\tau_{\sigma\beta}(G^\delta_{\gamma\tau}R_{\alpha\delta})\notag\\
			&+\Gamma^\tau_{\sigma\alpha}(G^\delta_{\gamma\beta}R_{\tau\delta})+\Gamma^\tau_{\sigma\gamma}(G^\delta_{\tau\alpha}F_{\delta\beta})+\Gamma^\tau_{\sigma\alpha}(G^\delta_{\gamma\tau}F_{\delta\beta})+\Gamma^\tau_{\sigma\beta}(G^\delta_{\gamma\alpha}F_{\delta\tau})\notag\\
			&+\Gamma^\tau_{\sigma\gamma}(G^\delta_{\tau\beta}F_{\alpha\delta})+\Gamma^\tau_{\sigma\beta}(G^\delta_{\gamma\tau}F_{\alpha\delta})+\Gamma^\tau_{\sigma\alpha}(G^\delta_{\gamma\beta}F_{\tau\delta}).\notag
		\end{align}

	Now we give an estimate of $\abs{\partial_{\delta} G^\alpha_{\beta\gamma}}$ in terms of suitable $C$ and $\Theta$, where $C$ is a constant and $\Theta=\Theta(1/k,r,p)$ denotes a continuous function depending on $1/k,r$ and $p$; to do so we use the bounds on $\eta_{\rho},\eta_{\rho\rho},\eta_{ij}$ and on their partial derivatives in $\rho$ together with
	\begin{align}\label{stime christoffel}
		\abs{\Gamma^\rho_{ij}}\leq C\rho;&& \abs{\Gamma^i_{\rho j}}\leq C\rho;&& \abs{\Gamma^i_{jt}}\leq C\rho^2; &&\Gamma^{\rho}_{\rho i}=\Gamma^i_{\rho\rho}=\Gamma^{\rho}_{\rho\rho}=0.
	\end{align}
	Then, using \eqref{stime eta pt1} and \eqref{estimeta} we have 
	\begin{align*}
		\abs{\partial_\rho G^\rho_{ij}}&=\abs{\frac{k^2\eta_{\rho\rho}\eta_{ij}}{1+k^2\eta_\rho^2}+\frac{k^2\eta_{\rho}\partial_{\rho}\eta_{ij}}{1+k^2\eta_\rho^2}-\frac{k^42\eta_{\rho}^2\eta_{\rho\rho}\eta_{ij}}{(1+k^2\eta_\rho^2)^2}}\\
		&\leq\frac{Ck^2}{r^2\frac{k^2}{r^2}\abs{\frac{r^2}{k^2}-\frac{\eta^{-2}}{4}(y')^2}}+\frac{Ck^2}{r^2\frac{k^2}{r^2}\abs{\frac{r^2}{k^2}-\frac{\eta^{-2}}{4}(y')^2}}+\frac{Ck^4}{r^4\frac{k^4}{r^4}\abs{\frac{r^2}{k^2}-\frac{\eta^{-2}}{4}(y')^{2}}^2}\notag\\
		&\leq C\Theta+C\Theta+C\Theta\leq C\Theta;\\
		\abs{\partial_lG^\rho_{ij}}&=\abs{\frac{k^2\eta_\rho\partial_l\eta_{ij}}{1+k^2\eta_\rho^2}}\leq\abs{\frac{k^2r\eta_{\rho}^2}{k^2\eta_{\rho}^2}}\leq Cr\leq C;\\
		\abs{\partial_{\rho}G^\rho_{\rho\rho}}&=\abs{\frac{k^2\eta_{\rho\rho}^2}{1+k^2\eta_{\rho}^2}+\frac{k^2\eta_{\rho}\eta_{\rho\rho\rho}}{1+k^2\eta_{\rho}^2}-\frac{2k^4\eta_{\rho}^2\eta_{\rho\rho}^2}{(1+k^2\eta_{\rho}^2)^2}}\\
		&\leq	\frac{k^2C}{r^4\frac{k^2}{r^2}\abs{\frac{r^2}{k^2}-\frac{\eta^{-2}}{4}(y')}}+\frac{k^2C}{r^4\frac{k^2}{r^2}\abs{\frac{r^2}{k^2}-\frac{\eta^{-2}}{4}(y')}}+\frac{k^4C}{r^6\frac{k^4}{r^4}\pa{\frac{r^2}{k^2}-\frac{\eta^{-2}}{4}(y')}^2}\notag\\
		&\leq\frac{C}{r^2}\Theta+\frac{C}{r^2}\Theta+\frac{C}{r^2}\Theta\leq C\Theta;\\
		\partial_lG^\rho_{\rho\rho}&=\partial_l G^i_{jt}=\partial_\rho G^i_{jt}=\partial_\rho G^i_{j\rho}=\partial_lG^i_{j\rho}=\partial_\rho G^i_{\rho\rho}=\partial_l G^i_{\rho\rho}=\partial_{\rho}G^\rho_{\rho i}=\partial_lG^\rho_{\rho i}=0.
	\end{align*}
	We start computing $\abs{\partial_{\gamma}F_{\alpha\beta}}$ and $\abs{\nabla_{\gamma}F_{\alpha\beta}}$: we use \eqref{stime eta pt1} and \eqref{estimeta} in order to obtain
	\begin{align}\label{partial F}
		\abs{\partial_{\rho}F_{\rho\rho}}=&\abs{\frac{k^2\eta_{\rho\rho\rho}\eta^i_i}{1+k^2\eta_{\rho}^2}+\frac{k^2\eta_{\rho\rho}\partial_{\rho}\eta^i_i}{1+k^2\eta_{\rho}^2}-\frac{2k^4\eta_{\rho}\eta_{\rho\rho}^2\eta^i_i}{(1+k^2\eta_{\rho}^2)^2}}\\
		\leq&\frac{C}{r}\Theta+\frac{C}{r}\Theta+\frac{C}{r}\Theta\leq C\Theta;\notag\\
		\abs{\partial_{l}F_{\rho\rho}}=&\abs{\frac{k^2\eta_{\rho\rho}\partial_l\eta^i_i}{1+k^2\eta_{\rho}^2}}\leq\frac{k^2C}{r^2\frac{k^2}{r^2}\abs{\frac{r^2}{k^2}+\frac{\eta^{-2}}{4}(y')^2}}=C\Theta;\notag\\
		\abs{\partial_{\rho}F_{ij}}=&
		\Bigg|-\frac{k^2\eta_{\rho}\eta_{\rho\rho}L}{1+k^2\eta_{\rho}^2}-\frac{k^2\eta_{\rho}^2\partial_{\rho} L}{1+k^2\eta_{\rho}^2}+\frac{2k^4\eta_{\rho}^3\eta_{\rho\rho}L}{(1+k^2\eta_{\rho}^2)^2}+\frac{Ck^2(\partial_{\rho}\eta_{ANG}\eta_{ANG})_{ij}}{1+k^2\eta_{\rho}^2}+\frac{Ck^4\eta_{\rho}\eta_{\rho\rho}(\eta_{ANG}\eta_{ANG})_{ij}}{(1+k^2\eta_{\rho}^2)^2}\notag\\
		&+\frac{k^2\eta_{\rho\rho\rho}\eta_{ij}}{(1+k^2\eta_{\rho}^2)^2}+\frac{k^2\eta_{\rho\rho}\partial_{\rho}\eta_{ij}}{(1+k^2\eta_{\rho}^2)^2}-\frac{8k^4\eta_{\rho}\eta_{\rho\rho}^2\eta_{ij}}{(1+k^2\eta_{\rho}^2)^3}\Bigg|\notag\\	
		=&\Bigg|-\frac{ k^2\eta_{\rho}\eta_{\rho\rho}L}{(1+k^2\eta_{\rho}^2)^2}-\frac{ k^2\eta_{\rho}^2\partial_{\rho} L}{1+k^2\eta_{\rho}^2}+\frac{C k^2(\partial_{\rho}\eta_{ANG}\eta_{ANG})_{ij}}{1+k^2\eta_{\rho}^2}+\frac{ Ck^4\eta_{\rho}\eta_{\rho\rho}(\eta_{ANG}\eta_{ANG})_{ij}}{(1+k^2\eta_{\rho}^2)^2}\notag\\
		&+\frac{ k^2\eta_{\rho\rho\rho}\eta_{ij}}{(1+k^2\eta_{\rho}^2)^2}+\frac{ k^2\eta_{\rho\rho}\partial_{\rho}\eta_{ij}}{(1+k^2\eta_{\rho}^2)^2}-\frac{ 8k^4\eta_{\rho}\eta_{\rho\rho}^2\eta_{ij}}{(1+k^2\eta_{\rho}^2)^3}\Bigg|\notag\\
		\leq& \frac{C}{k^2}\Theta+C+C\Theta+C\Theta+\frac{C}{k^2}\Theta+\frac{C}{k^2}\Theta+\frac{C}{k^2}\Theta\leq C\pa{1+\Theta}
		\notag\\
		\abs{\partial_{l}F_{ij}}=&\abs{-\frac{k^2\eta_{\rho}^2\partial_lR_{i\rho j\rho}}{1+k^2\eta_{\rho}^2}+\frac{Ck^2(\partial_l\eta_{ANG}\eta_{ANG})_{ij}}{1+k^2\eta_{\rho}^2}+\frac{k^2\eta_{\rho\rho}\partial_l\eta_{ij}}{(1+k^2\eta_{\rho}^2)^2}}\leq C+C+\frac{C}{k^2}\Theta\leq C\pa{1+\frac{1}{k^2}\Theta};\notag\\
		\partial_\rho F_{i\rho}=&\partial_{l}F_{i\rho}=0.\notag
	\end{align}
	Note that in the above estimates we have not used \eqref{eta rho rho e eta rho} and \eqref{eta ang ed eta rho}, although it is useful to exploit them and $\abs{\eta_{ij}}\leq Cr\abs{\eta_{\rho}},\, \abs{\partial_l\eta_{ij}}\leq Cr\abs{\eta_{\rho}}$ (see \eqref{stime eta}) to estimate $\abs{\partial_l F_{\rho\rho}}$ and $\abs{\partial_{\rho}F_{ij}}$:
	\begin{align}\label{oter est partial F}
		\begin{cases}
			\abs{\partial_l F_{\rho\rho}}&\leq\frac{C}{r-\rho};\\
			\abs{\partial_{\rho}F_{ij}}&\leq C+\frac{C}{r-\rho}+\frac{C}{k^2}\Theta.
		\end{cases}
	\end{align}
	Combining \eqref{resti bar mod}, \eqref{stime christoffel} with \eqref{partial F} we deduce
	\begin{align}\label{nabla F}
		\abs{\nabla_{\rho}F_{\rho\rho}}&=\abs{\partial_{\rho}F_{\rho\rho}}\leq C\Theta;\\
		\abs{\nabla_{l}F_{\rho\rho}}&=\abs{\partial_{l}F_{\rho\rho}}\leq C\Theta;\notag\\
		\abs{\nabla_{\rho}F_{ij}}&=\abs{\partial_{\rho}F_{ij}-\Gamma^l_{\rho i}F_{lj}-\Gamma^l_{\rho j}F_{il}}\notag\\
		&\leq \abs{\partial_{\rho}F_{ij}}+\abs{\Gamma^l_{\rho i}F_{lj}}+\abs{\Gamma^l_{\rho j}F_{il}}\notag\\
		&\leq \abs{\partial_{\rho}F_{ij}}+\abs{\Gamma^l_{\rho i}}\abs{F_{lj}}+\abs{\Gamma^l_{\rho j}}\abs{F_{il}}\notag\\
		&\leq C\pa{1+\Theta}+ C\rho\pa{1+\frac{1}{k^2}}+ C\rho\pa{1+\frac{1}{k^2}}\leq C\pa{1+\Theta};\notag\\
		\abs{\nabla_{l}F_{ij}}&=\abs{\partial_{l}F_{ij}-\Gamma^t_{li}F_{tj}-\Gamma^t_{lj}F_{it}}\notag\\
		&\leq\abs{\partial_{l}F_{ij}}+\abs{\Gamma^t_{li}F_{tj}}+\abs{\Gamma^t_{lj}F_{it}}\notag\\
		&\leq \abs{\partial_{l}F_{ij}}+\abs{\Gamma^t_{li}}\abs{F_{tj}}+\abs{\Gamma^t_{lj}}\abs{F_{it}}\notag\\
		&\leq C\pa{1+\frac{1}{k^2}\Theta}+ C\rho\pa{1+\frac{1}{k^2}\Theta}+ C\rho\pa{1+\frac{1}{k^2}\Theta}\leq C\pa{1+\frac{1}{k^2}\Theta};\notag\\
		\abs{\nabla_{l}F_{i\rho}}&=\abs{-\Gamma^\rho_{ij}F_{\rho\rho}-\Gamma^l_{j\rho}F_{il}}\notag\\
		&\leq\abs{\Gamma^\rho_{ij}}\abs{F_{\rho\rho}}+\abs{\Gamma^l_{j\rho}}\abs{F_{il}}\notag\\
		&\leq C\rho\Theta+C\rho\pa{1+\frac{1}{k^2}\Theta}\leq C\pa{1+\Theta};\notag\\
		\nabla_{\rho}F_{i\rho}&=0.\notag
	\end{align}
	Note that using \eqref{oter est partial F}, we obtain
	\begin{align}\label{nabla F 2}
		\begin{cases}
			\abs{\nabla_{l}F_{\rho\rho}}\leq\frac{C}{r-\rho};\\
			\abs{\nabla_{\rho}F_{ij}}\leq C+\frac{C}{r-\rho}+\frac{C}{k^2}\Theta;\\
			\abs{\nabla_lF_{i\rho}}\leq C+\frac{C}{r-\rho}+\frac{C}{k^2}\Theta.
		\end{cases}
			\end{align} 
	We are now ready to compute $\abs{\partial_{\gamma}\partial_{\delta}F_{\alpha\beta}}$ and $\abs{\nabla_{\gamma}\nabla_{\delta}F_{\alpha\beta}}$: using \eqref{stime eta pt1}, \eqref{estimeta} and the definition of $F_{\alpha\beta}$ we get
	\begin{align}\label{derivata seconda}
		\abs{\partial_{\rho}\partial_\rho F_{\rho\rho}}=&\Bigg|\frac{k^2\eta_{\rho\rho\rho\rho}\eta^i_i}{1+k^2\eta_{\rho}^2}+\frac{k^2\eta_{\rho\rho\rho}\partial_\rho\eta^i_i}{1+k^2\eta_{\rho}^2}-\frac{2k^4\eta_{\rho}\eta_{\rho\rho}\eta_{\rho\rho\rho}\eta^i_i}{(1+k^2\eta_{\rho}^2)^2}+\frac{k^2\eta_{\rho\rho\rho}\partial_{\rho}\eta^i_i}{1+k^2\eta_{\rho}^2}\\
		&+\frac{k^2\eta_{\rho\rho}\partial_{\rho}\partial_{\rho}\eta^i_i}{1+k^2\eta_{\rho}^2}-\frac{2k^4\eta_{\rho}\eta_{\rho\rho}^2\partial_\rho\eta^i_i}{(1+k^2\eta_{\rho}^2)^2}-\frac{2k^4\eta_{\rho\rho}^3\eta^i_i}{(1+k^2\eta_{\rho}^2)^2}\notag\\
		&-\frac{4k^4\eta_{\rho}\eta_{\rho\rho}\eta_{\rho\rho\rho}\eta^i_i}{(1+k^2\eta_{\rho}^2)^2}-\frac{2k^4\eta_{\rho}\eta_{\rho\rho}^2\partial_\rho\eta^i_i}{(1+k^2\eta_{\rho}^2)^2}+\frac{8k^6\eta_{\rho}^2\eta_{\rho\rho}^3\eta^i_i}{(1+k^2\eta_{\rho}^2)^3}\Bigg|\notag\\
		\leq& C\Theta+C\Theta+C\Theta+C\Theta+C\Theta+C\Theta+C\Theta+C\Theta+C\Theta+C\Theta\leq C\Theta\notag\\
		\abs{\partial_t\partial_l F_{\rho\rho}}=&\abs{\frac{k^2\eta_{\rho\rho}\partial_t\partial_{l}\eta^i_i}{1+k^2\eta_{\rho}^2}}\leq C\Theta;\notag\\
		\abs{\partial_{\rho}\partial_\rho F_{ij}}=&\Bigg|\frac{k^2\eta_{\rho\rho}^2L}{1+k^2\eta_{\rho}^2}+\frac{k^2\eta_{\rho}\eta_{\rho\rho\rho}L}{1+k^2\eta_{\rho}^2}+\frac{k^4\eta_{\rho}^2\eta_{\rho\rho}^2L}{(1+k^2\eta_{\rho}^2)^2}+\frac{k^2\eta_{\rho}\eta_{\rho\rho}\partial_\rho L}{1+k^2\eta_{\rho}^2}+\frac{k^2\eta_{\rho}^2\partial_{\rho}\partial_{\rho}L}{1+k^2\eta_{\rho}^2}\notag\\
		&+\frac{k^4\eta_{\rho}^3\eta_{\rho\rho}\partial_\rho L}{(1+k^2\eta_{\rho}^2)^2}+\frac{k^4\eta_{\rho}^2\eta_{\rho\rho}^2L}{(1+k^2\eta_{\rho}^2)^2}+\frac{k^4\eta_{\rho}^3\eta_{\rho\rho\rho}L}{(1+k^2\eta_{\rho}^2)^2}+\frac{k^6\eta_{\rho}^4\eta_{\rho\rho}^2L}{(1+k^2\eta_{\rho}^2)^3}+\frac{k^4\eta_{\rho}^2\eta_{\rho\rho}^2\partial_{\rho}L}{(1+k^2\eta_{\rho}^2)^2}\notag\\
		&+\frac{Ck^2(\partial_\rho\partial_{\rho}\eta_{ANG}\eta_{ANG})_{ij}}{1+k^2\eta_{\rho}^2}+\frac{Ck^2(\partial_{\rho}\eta_{ANG}\partial_{\rho}\eta_{ANG})_{ij}}{1+k^2\eta_{\rho}^2}-\frac{2Ck^4\eta_{\rho}\eta_{\rho\rho}(\eta_{ANG}\partial_{\rho}\eta_{ANG})_{ij}}{(1+k^2\eta_{\rho}^2)^2}\notag\\
		&+\frac{Ck^4\eta_{\rho\rho}^2(\eta_{ANG}\eta_{ANG})_{ij}}{(1+k^2\eta_{\rho}^2)^2}+\frac{Ck^4\eta_{\rho}\eta_{\rho\rho\rho}(\eta_{ANG}\eta_{ANG})_{ij}}{(1+k^2\eta_{\rho}^2)^2}+\frac{Ck^4\eta_{\rho}\eta_{\rho\rho}(\eta_{ANG}\partial_{\rho}\eta_{ANG})_{ij}}{(1+k^2\eta_{\rho}^2)^2}\notag\\
		&+\frac{Ck^6\eta_{\rho}^2\eta_{\rho\rho}^2(\eta_{ANG}\eta_{ANG})_{ij}}{(1+k^2\eta_{\rho}^2)^3}+\frac{k^2\eta_{\rho\rho\rho\rho}\eta_{ij}}{(1+k^2\eta_{\rho}^2)^2}+\frac{k^2\eta_{\rho\rho\rho}\partial_{\rho}\eta_{ij}}{(1+k^2\eta_{\rho}^2)^2}-\frac{2k^4\eta_{\rho}\eta_{\rho\rho}\eta_{\rho\rho\rho}\eta_{ij}}{(1+k^2\eta_{\rho}^2)^3}\notag\\
		&+\frac{k^2\eta_{\rho\rho\rho}\partial_{\rho}\eta_{ij}}{(1+k^2\eta_{\rho}^2)^2}+\frac{k^2\eta_{\rho\rho}\partial_{\rho}\partial_{\rho}\eta_{ij}}{(1+k^2\eta_{\rho}^2)^2}-\frac{2k^4\eta_{\rho}\eta_{\rho\rho}^2\partial_{\rho}\eta_{ij}}{(1+k^2\eta_{\rho}^2)^3}+\frac{Ck^4\eta_{\rho\rho}^3\eta_{ij}}{(1+k^2\eta_{\rho}^2)^3}\notag\\
		&+\frac{Ck^4\eta_{\rho}\eta_{\rho\rho}\eta_{\rho\rho\rho}\eta_{ij}}{(1+k^2\eta_{\rho}^2)^3}+\frac{Ck^4\eta_{\rho}\eta_{\rho\rho}^2\partial_{\rho}\eta_{ij}}{(1+k^2\eta_{\rho}^2)^3}+\frac{Ck^6\eta_{\rho}^2\eta_{\rho\rho}^3\eta_{ij}}{(1+k^2\eta_{\rho}^2)^4}\Bigg|\notag\\
		\leq& C\Theta+C\Theta+C\Theta+C\Theta+C+C\Theta+C\Theta+C\Theta+C\Theta+C\Theta+C\Theta\notag\\
		&+C\Theta+C\Theta+C\Theta+C\Theta+C\Theta+C\Theta+\frac{C}{k^2}\Theta+\frac{C}{k^2}\Theta+\frac{C}{k^2}\Theta+\frac{C}{k^2}\Theta
		\notag\\
		&+\frac{C}{k^2}\Theta+\frac{C}{k^2}\Theta+\frac{C}{k^2}\Theta+\frac{C}{k^2}\Theta+\frac{C}{k^2}\Theta+\frac{C}{k^2}\Theta\leq C\pa{1+\Theta};
		\notag\\
		\abs{\partial_t\partial_l F_{ij}}=&\Bigg|-\frac{k^2\eta_{\rho}^2\partial_t\partial_lR_{i\rho j\rho}}{1+k^2\eta_{\rho}^2}+\frac{2k^2(\partial_t\partial_l\eta_{ANG}\eta_{ANG})_{ij}}{(1+k^2\eta_{\rho}^2)^2}+\frac{2k^2\eta_{\rho\rho}^2(\partial_t\partial_l\eta_{ANG}\eta_{ANG})_{ij}}{(1+k^2\eta_{\rho}^2)^2}\notag\\
		&+\frac{2k^2(\partial_t\eta_{ANG}\partial_l\eta_{ANG})_{ij}}{(1+k^2\eta_{\rho}^2)^2}
		+\frac{2k^2\eta_{\rho\rho}^2(\partial_t\eta_{ANG}\partial_l\eta_{ANG})_{ij}}{(1+k^2\eta_{\rho}^2)^2}\Bigg|\notag\\
		&\leq C+C+\frac{C}{k^2}\Theta+C+\frac{C}{k^2}\Theta\notag\\
		&\leq C\pa{1+\frac{1}{k^2}\Theta};\notag\\
		\abs{\partial_{\rho}\partial_\rho F_{i\rho}}=&\abs{\partial_t\partial_l F_{i\rho}}=0.\notag
	\end{align}
	Using \eqref{resti bar mod}, \eqref{stime christoffel}, \eqref{partial F}, \eqref{nabla F} and \eqref{derivata seconda}, we deduce
	\begin{align}\label{nabla nabla F}
		\abs{\nabla_{\rho}\nabla_\rho F_{\rho\rho}}=&\abs{\partial_{\rho}\partial_\rho F_{\rho\rho}}\leq C\Theta;\\
		\abs{\nabla_t\nabla_l F_{\rho\rho}}=&\abs{\partial_t\nabla_l F_{\rho\rho}-\Gamma^\rho_{tl}(\nabla_{\rho}F_{\rho\rho})-\Gamma^i_{tl}(\nabla_{i}F_{\rho\rho})}=\abs{\partial_t\partial_l F_{\rho\rho}-\Gamma^\rho_{tl}(\nabla_{\rho}F_{\rho\rho})-\Gamma^i_{tl}(\nabla_{i}F_{\rho\rho})}\notag\\
		\leq& C\Theta+ C\Theta+\rho^2C\Theta\leq C\Theta;\notag\\
		\abs{\nabla_\rho\nabla_\rho F_{ij}}=&\abs{\partial_\rho\nabla_\rho F_{ij}-\Gamma^\alpha_{\rho i}\nabla_\rho F_{\alpha j}-\Gamma^\alpha_{\rho j}\nabla_{\rho}F_{i\alpha}}\notag\\
		=&\abs{\partial_\rho\pa{\partial_\rho F_{ij}-\Gamma^l_{\rho i}F_{lj}-\Gamma^l_{\rho j}F_{il}}-\Gamma^l_{\rho i}\nabla_\rho F_{lj}-\Gamma^l_{\rho j}\nabla_{\rho}F_{il}}\notag\\
		=&\abs{\partial_{\rho}\partial_{\rho}F_{ij}-\partial_{\rho}\Gamma^l_{\rho i}F_{lj}+\Gamma^l_{\rho i}\partial_{\rho}F_{lj}-\partial_{\rho}\Gamma^l_{\rho j}F_{il}+\Gamma^l_{\rho j}\partial_{\rho}F_{il}-\Gamma^l_{\rho i}\nabla_\rho F_{lj}-\Gamma^l_{\rho j}\nabla_{\rho}F_{il}}\notag\\
		\leq& C\pa{1+\Theta}+C\pa{1+\frac{1}{k^2}\Theta}+ C\rho\pa{1+\frac{1}{k^2}\Theta}+C\pa{1+\frac{1}{k^2}\Theta}\notag\\
		&+ C\rho\pa{1+\frac{1}{k^2}\Theta}+C\rho\pa{1+\Theta}+ C\rho\pa{1+\Theta}       \notag\\
		\leq&C\pa{1+\Theta};\notag\\
		\abs{\nabla_t\nabla_l F_{ij}}=&\big|\partial_t\nabla_lF_{ij}-\Gamma^s_{ti}(\nabla_l F_{sj})-\Gamma^s_{tj}(\nabla_l F_{si})-\Gamma^s_{tl}(\nabla_s F_{ij})\notag\\
		&-\Gamma^\rho_{tl}(\nabla_\rho F_{ij})-\Gamma^\rho_{ti}(\nabla_lF_{\rho j})-\Gamma^\rho_{tj}(\nabla_lF_{i\rho})\big|\notag\\
		=&\big|\partial_t(\partial_{l}F_{ij}-\Gamma^s_{li}F_{sj}-\Gamma^s_{lj}F_{is})
		-\Gamma^s_{ti}(\nabla_l F_{sj})-\Gamma^s_{tj}(\nabla_l F_{si})\notag\\
		&-\Gamma^s_{tl}(\nabla_s F_{ij})-\Gamma^\rho_{tl}(\nabla_\rho F_{ij})-\Gamma^\rho_{ti}(\nabla_lF_{\rho j})-\Gamma^\rho_{tj}(\nabla_lF_{i\rho})\big|\notag\\
		=&\big|\partial_t\partial_{l}F_{ij}-\partial_t\Gamma^s_{li}F_{sj}-\Gamma^s_{li}\partial_tF_{sj}-\partial_t\Gamma^s_{lj}F_{is}-\Gamma^s_{lj}\partial_tF_{is}\notag\\
		&-\Gamma^s_{ti}(\nabla_l F_{sj})-\Gamma^s_{tj}(\nabla_l F_{si})-\Gamma^s_{tl}(\nabla_s F_{ij})\notag\\
		&-\Gamma^\rho_{tl}(\nabla_\rho F_{ij})-\Gamma^\rho_{ti}(\nabla_lF_{\rho j})-\Gamma^\rho_{tj}(\nabla_lF_{i\rho})\big|\notag\\
		\leq&
		\abs{\partial_t\partial_{l}F_{ij}}+\abs{\partial_t\Gamma^s_{li}F_{sj}}+\abs{\Gamma^s_{li}\partial_tF_{sj}}+\abs{\partial_t\Gamma^s_{lj}F_{is}}+\abs{\Gamma^s_{lj}\partial_tF_{is}}\notag\\
		&+\abs{\Gamma^s_{ti}(\nabla_l F_{sj})}+\abs{\Gamma^s_{tj}(\nabla_l F_{si})}+\abs{\Gamma^s_{tl}(\nabla_s F_{ij})}\notag\\
		&+\abs{\Gamma^\rho_{tl}(\nabla_\rho F_{ij})}+\abs{\Gamma^\rho_{ti}(\nabla_lF_{\rho j})}+\abs{\Gamma^\rho_{tj}(\nabla_lF_{i\rho})}\notag\\
		\leq& C\pa{1+\frac{1}{k^2}\Theta}+C\pa{1+\frac{1}{k^2}\Theta}+ C\rho\pa{1+\frac{1}{k^2}\Theta}\notag\\
		&+C\rho\pa{1+\frac{1}{k^2}\Theta}+\rho^2 C\pa{1+\frac{1}{k^2}\Theta}+\rho^2\pa{C+\frac{C}{k^2}\Theta}\notag\\
		&+\abs{\Gamma^\rho_{ll}(\nabla_\rho F_{ij})}+\abs{\Gamma^\rho_{li}(\nabla_lF_{\rho j})}+\abs{\Gamma^\rho_{lj}(\nabla_lF_{i\rho})}\notag\\
		\abs{\nabla_\rho\nabla_\rho F_{i\rho}}=&\abs{\partial_\rho\nabla_\rho F_{i\rho}-\Gamma^\alpha_{\rho i}\nabla_{\rho}F_{\alpha\rho}-\Gamma^\alpha_{\rho\rho}\nabla_\rho F_{i\alpha}-\Gamma^\alpha_{\rho\rho}\nabla_\alpha F_{i\rho}}=0;\notag\\
		\abs{\nabla_t\nabla_l F_{i\rho}}=&\abs{\partial_t\nabla_l F_{i\rho}-\Gamma^\alpha_{t i}\nabla_{\rho}F_{\alpha\rho}-\Gamma^\alpha_{t\rho}\nabla_\rho F_{i\alpha}-\Gamma^\alpha_{tl}\nabla_\alpha F_{i\rho}};\notag\\
		=&\abs{\partial_t\pa{\partial_lF_{i\rho}-\Gamma^\rho_{li}F_{\rho\rho}-\Gamma^j_{l\rho}F_{ij}}-\Gamma^\alpha_{t i}\nabla_{\rho}F_{\alpha\rho}-\Gamma^\alpha_{t\rho}\nabla_\rho F_{i\alpha}-\Gamma^\alpha_{tl}\nabla_\alpha F_{i\rho}}\notag\\
		=&\abs{\partial_t\Gamma^\rho_{li}F_{\rho\rho}-\Gamma^\rho_{li}\partial_tF_{\rho\rho}-\partial_t\Gamma^j_{l\rho}F_{ij}-\Gamma^j_{l\rho}\partial_tF_{ij}-\Gamma^j_{ti}\nabla_{l}F_{j\rho}-\Gamma_{ti}^\rho\nabla_{l}F_{\rho\rho}-\Gamma^j_{t\rho}\nabla_{l}F_{ij}-\Gamma^j_{tl}\nabla_j F_{i\rho}}\notag\\
		\leq & \abs{\partial_t\Gamma^\rho_{li}F_{\rho\rho}}+\abs{\Gamma^\rho_{li}\partial_tF_{\rho\rho}}+C\pa{1+\frac{1}{k^2}\Theta}+ C\rho\pa{1+\frac{1}{k^2}\Theta}+\abs{\Gamma^j_{ti}\nabla_{l}F_{j\rho}}\notag\\
		&+\abs{\Gamma_{ti}^\rho\nabla_{l}F_{\rho\rho}}+ C\rho\pa{1+\frac{1}{k^2}\Theta}+\abs{\Gamma^j_{tl}\nabla_j F_{i\rho}}; \notag
	\end{align}
	To obtain a final estimate for $\abs{\nabla_t\nabla_lF_{ij}}$ and $\abs{\nabla_t\nabla_lF_{i\rho}}$, we use equations \eqref{F rho rho alternative} and \eqref{nabla F 2} to finally deduce:
	\begin{align}\label{nab tl ij irho}
		\abs{\nabla_t\nabla_lF_{ij}}\leq C+\frac{C}{r-\rho}+\frac{C}{k^2}\Theta,\qquad \abs{\nabla_t\nabla_lF_{i\rho}}\leq C+\frac{C}{r-\rho}+\frac{C}{k^2}\Theta.
	\end{align}
	We are now ready to give an estimate of the components of the remainders of $\bar{\Delta}\ol{\ric}$:
	\begin{align}\label{components of lap ricci}
		\abs{\ol{\nabla}_{\rho}\ol{\nabla}_\rho\ol{R}_{\rho\rho}}=&\big|\nabla_\rho\nabla_\rho R_{\rho\rho}+\nabla_\rho\nabla_\rho F_{\rho\rho}-3G^\rho_{\rho\rho}\nabla_{\rho}R_{\rho\rho}\\
		&-3G^\rho_{\rho\rho}\nabla_{\rho}F_{\rho\rho}-2\partial_\rho G^\rho_{\rho\rho}R_{\rho\rho}-2\partial_\rho G^\rho_{\rho\rho}F_{\rho\rho}\big|\notag\\
		&\leq C+ C\Theta+ C\Theta+C\Theta+C\Theta+C\Theta\leq C\pa{1+\Theta}\notag;\\
		\abs{\ol{\nabla}_{t}\ol{\nabla}_l\ol{R}_{\rho\rho}}=&\big|\nabla_t\nabla_l R_{\rho\rho}+\nabla_t\nabla_{l}F_{\rho\rho}-G^\rho_{tl}\nabla_{\rho}R_{\rho\rho}-G^\rho_{tl}\nabla_{\rho}F_{\rho\rho}+2\Gamma^\rho_{tl}\pa{G^\rho_{\rho\rho}R_{\rho\rho}}\notag\\
		&+2\Gamma^s_{t\rho}\pa{G^\rho_{ls}R_{\rho\rho}}+2\Gamma^\rho_{tl}\pa{G^\rho_{\rho\rho}F_{\rho\rho}}+2\Gamma^s_{t\rho}\pa{G^\rho_{ls}F_{\rho\rho}}\big|\notag\\
		&\leq C+C\Theta+C+C\Theta+C\Theta+C+C\Theta+C\Theta\leq C\pa{1+\Theta};\notag\\
		\abs{\ol{\nabla}_{\rho}\ol{\nabla}_\rho\ol{R}_{ij}}=&\big|\nabla_\rho\nabla_\rho R_{ij}+\nabla_\rho\nabla_\rho F_{ij}-G^\rho_{\rho\rho}\nabla_\rho R_{ij}-G^\rho_{\rho\rho}\nabla_\rho F_{ij}\big|\notag\\
		\leq& C+C\pa{1+\Theta}+C\Theta+C\Theta\leq C\pa{1+\Theta};\notag\\
		\abs{\ol{\nabla}_{t}\ol{\nabla}_l\ol{R}_{ij}}=&\big|\nabla_t\nabla_l R_{ij}+\nabla_t\nabla_lF_{ij}-G^\rho_{ti}\nabla_lR_{\rho j}-G^\rho_{tj}\nabla_lR_{i\rho}-G^\rho_{tl}\nabla_\rho R_{ij}\notag\\
		&-G^\rho_{ti}\nabla_lF_{\rho j}-G^\rho_{tj}\nabla_lF_{i\rho}-G^\rho_{tl}\nabla_\rho F_{ij}\notag\\
		&-\partial_{t}G^\rho_{li}R_{\rho j}-\partial_{t}G^\rho_{lj}R_{i\rho}-G^\rho_{li}\partial_tR_{\rho j}-G^\rho_{lj}\partial_tR_{i\rho}\notag\\
		&+\Gamma_{tl}^s\pa{G^\rho_{si} R_{\rho j}}+\Gamma^s_{ti}\pa{G^\rho_{ls}R_{\rho j}}+\Gamma^\rho_{tj}\pa{G^\rho_{li}R_{\rho\rho}}+\Gamma^s_{tj}\pa{G^\rho_{li}R_{\rho s}}+\Gamma^s_{tl}\pa{G^\rho_{sj}R_{i\rho}}\notag\\
		&+\Gamma^s_{tj}\pa{G^\rho_{ls}R_{i\rho}}+\Gamma^\rho_{ti}\pa{G^\rho_{lj}R_{\rho\rho}}+\Gamma^s_{ti}\pa{G^\rho_{lj}R_{\rho s}}+\Gamma^\rho_{tj}\pa{G^\rho_{lj}F_{\rho\rho}}+\Gamma^\rho_{ti}\pa{G^\rho_{li}F_{\rho\rho}}\big|\notag\\
		&\leq C+ C\pa{1+\frac{1}{r-\rho}+\frac{1}{k^2}\Theta}+C+C+C+C\pa{1+\frac{1}{r-\rho}+\frac{1}{k^2}\Theta}\notag\\
		&+C\pa{1+\frac{1}{r-\rho}+\frac{1}{k^2}\Theta}+C\pa{1+\frac{1}{r-\rho}+\frac{1}{k^2}\Theta}+C\pa{1+\frac{1}{r-\rho}+\frac{1}{k^2}}\notag\\
		&+C+C+C+C+C+C+C+C+C+C+C+C+\frac{C}{r-\rho}+\frac{C}{r-\rho}\notag\\
		&\leq C+\frac{C}{k^2}\Theta+\frac{C}{r-\rho};\notag\\
		\abs{\ol{\nabla}_{\rho}\ol{\nabla}_\rho\ol{R}_{i\rho}}=&\big|\nabla_\rho\nabla_{\rho}R_{i\rho}-2G^\rho_{\rho\rho}\nabla_{\rho}R_{i\rho}-\partial_{\rho}G^\rho_{\rho\rho}R_{i\rho}-G^\rho_{\rho\rho}\partial_\rho R_{i\rho}+\Gamma^t_{\rho i}\pa{G^\rho_{\rho\rho}R_{t\rho}}\big|\notag\\
		&\leq C+C\Theta+C\Theta+C\Theta+C\Theta\leq C\pa{1+\Theta};\notag\\
		\abs{\ol{\nabla}_{t}\ol{\nabla}_l\ol{R}_{i\rho}}=&\big|\nabla_t\nabla_lR_{i\rho}-\nabla_t\nabla_lF_{i\rho}-G^\rho_{ti}\nabla_lR_{\rho\rho}-G^\rho_{tl}\nabla_{\rho}R_{i\rho}\notag\\
		&-G^\rho_{ti}\nabla_lF_{\rho\rho}-\partial_tG^\rho_{li}R_{\rho \rho}-\partial_tG^\rho_{li}F_{\rho \rho}-G^\rho_{li}\partial_tR_{\rho\rho}-G^\rho_{li}\partial_tF_{\rho\rho}\notag\\
		&+\Gamma^s_{tl}\pa{G^\rho_{si}R_{\rho\rho}}+\Gamma^s_{ti}\pa{G^\rho_{ls}R_{\rho\rho}}+\Gamma^s_{t\rho}\pa{G^\rho_{li}R_{\rho s}}+\Gamma^\rho_{tl}\pa{G^\rho_{\rho\rho}R_{i\rho}}\notag\\
		&+\Gamma^s_{t\rho}\pa{G^\rho_{ls}R_{i\rho}}+\Gamma^s_{tl}\pa{G^\rho_{si}F_{\rho\rho}}+\Gamma^s_{ti}\pa{G^\rho_{ls}F_{\rho\rho}}\big|\notag\\
		\leq& C+C\pa{1+\frac{1}{r-\rho}+\frac{1}{k^2}\Theta}+C+C+\frac{C}{r-\rho}\notag\\
		&+C\frac{C}{r-\rho}+C+C+C+C+C+\frac{C}{r-\rho}+\frac{C}{r-\rho}\notag\\
		\leq& C\pa{1+\frac{1}{r-\rho}+\frac{1}{k^2}\Theta}.\notag
	\end{align}
	Note that we have used equations \eqref{normgbarg}, \eqref{resti bar mod}, \eqref{stime christoffel}, \eqref{partial F}, \eqref{nabla F} and \eqref{nabla nabla F} 
	to estimate $\abs{\ol{\nabla}_{\rho}\ol{\nabla}_\rho\ol{R}_{\rho\rho}},\,\abs{\ol{\nabla}_{t}\ol{\nabla}_l\ol{R}_{\rho\rho}}\,\abs{\ol{\nabla}_{\rho}\ol{\nabla}_\rho\ol{R}_{ij}}$ and $\abs{\ol{\nabla}_{\rho}\ol{\nabla}_\rho\ol{R}_{i\rho}}$ and we have used
	\eqref{normgbarg}, \eqref{resti bar mod},
	\eqref{F rho rho alternative}, \eqref{stime christoffel}, \eqref{oter est partial F}, \eqref{nabla F 2} and \eqref{nab tl ij irho} to estimate $\abs{\ol{\nabla}_{t}\ol{\nabla}_l\ol{R}_{ij}}$ and $\abs{\ol{\nabla}_t\ol{\nabla}_l\ol{R}_{i\rho}}$.
	Therefore, using the equations in \eqref{components of lap ricci} in \eqref{nabla nabla ricci} we conclude
	\begin{align}\label{lapRic}
		\abs{\bar{\Delta}\ol{\ric}}_{\ol{g}}=&\abs{\frac{1}{1+k^2\eta_{\rho}^2}\ol{\nabla}_\rho\ol{\nabla}_{\rho}(\ol{R}_{\rho\rho}+\ol{R}_{i\rho}+\ol{R}_{ij})+g^{lt}\ol{\nabla}_l\ol{\nabla}_t(\ol{R}_{\rho\rho}+\ol{R}_{i\rho}+\ol{R}_{ij})}_{\gb}\\
		\leq&\frac{1}{1+k^2\eta_{\rho}^2}\abs{\ol{\nabla}_\rho\ol{\nabla}_{\rho}\ol{R}_{\rho\rho}}_{\gb}+\frac{1}{1+k^2\eta_{\rho}^2}\abs{\ol{\nabla}_\rho\ol{\nabla}_{\rho}\ol{R}_{i\rho}}_{\gb}+\frac{1}{1+k^2\eta_{\rho}^2}\abs{\ol{\nabla}_\rho\ol{\nabla}_{\rho}\ol{R}_{ij}}_{\gb}\notag\\
		&+\abs{g^{lt}\ol{\nabla}_l\ol{\nabla}_t\ol{R}_{\rho\rho}}_{\gb}+\abs{g^{lt}\ol{\nabla}_l\ol{\nabla}_t\ol{R}_{i\rho}}_{\gb}+\abs{g^{lt}\ol{\nabla}_l\ol{\nabla}_t\ol{R}_{ij}}_{\gb}\notag\\
		=&\frac{1}{(1+k^2\eta_{\rho}^2)^2}\abs{\ol{\nabla}_\rho\ol{\nabla}_{\rho}\ol{R}_{\rho\rho}}_g+\frac{1}{(1+k^2\eta_{\rho}^2)^{\frac{3}{2}}}\abs{\ol{\nabla}_\rho\ol{\nabla}_{\rho}\ol{R}_{i\rho}}_g+\frac{1}{1+k^2\eta_{\rho}^2}\abs{\ol{\nabla}_\rho\ol{\nabla}_{\rho}\ol{R}_{ij}}_g\notag\\
		&+\frac{1}{1+k^2\eta_{\rho}^2}\abs{g^{lt}\ol{\nabla}_l\ol{\nabla}_t\ol{R}_{\rho\rho}}_g+\frac{1}{(1+k^2\eta_{\rho}^2)^{\frac{1}{2}}}\abs{g^{lt}\ol{\nabla}_l\ol{\nabla}_t\ol{R}_{i\rho}}_g+\abs{g^{lt}\ol{\nabla}_l\ol{\nabla}_t\ol{R}_{ij}}_g\notag\\
		\leq&\frac{1}{(1+k^2\eta_{\rho}^2)^2}\abs{\ol{\nabla}_\rho\ol{\nabla}_{\rho}\ol{R}_{\rho\rho}}+\frac{C}{(1+k^2\eta_{\rho}^2)^{\frac{3}{2}}}\abs{\ol{\nabla}_\rho\ol{\nabla}_{\rho}\ol{R}_{i\rho}}+\frac{C}{1+k^2\eta_{\rho}^2}\abs{\ol{\nabla}_\rho\ol{\nabla}_{\rho}\ol{R}_{ij}}\notag\\
		&+\frac{1}{1+k^2\eta_{\rho}^2}\abs{g^{lt}\ol{\nabla}_l\ol{\nabla}_t\ol{R}_{\rho\rho}}+\frac{C}{(1+k^2\eta_{\rho}^2)^{\frac{1}{2}}}\abs{g^{lt}\ol{\nabla}_l\ol{\nabla}_t\ol{R}_{i\rho}}+C\abs{g^{lt}\ol{\nabla}_l\ol{\nabla}_t\ol{R}_{ij}}\notag\\
		\leq&\frac{1}{(1+k^2\eta_{\rho}^2)^2}C(1+\Theta)+\frac{1}{(1+k^2\eta_{\rho}^2)^{\frac{3}{2}}}C(1+\Theta)+\frac{1}{1+k^2\eta_{\rho}^2}C(1+\Theta)\notag\\
		&+\frac{1}{1+k^2\eta_{\rho}^2}C(1+\Theta)+\frac{1}{(1+k^2\eta_{\rho}^2)^{\frac{1}{2}}}C\pa{1+\frac{1}{r-\rho}+\frac{1}{k^2}\Theta}+C\pa{1+\frac{1}{r-\rho}+\frac{1}{k^2}\Theta}\notag\\
		\leq& C\pa{1+\frac{1}{k^2}\Theta}+C\pa{1+\frac{1}{k^2}\Theta}+C\pa{1+\frac{1}{k^2}\Theta}\notag\\
		&+C\pa{1+\frac{1}{k^2}\Theta}+C\pa{1+\frac{1}{r-\rho}+\frac{1}{k^2}\Theta}+C\pa{1+\frac{1}{r-\rho}+\frac{1}{k^2}\Theta}\notag\\
		\leq& C\pa{1+\frac{1}{r-\rho}+\frac{1}{k^2}\Theta}.\notag
	\end{align}
	Now we compute $\abs{\partial_{\alpha}H}$ using estimates \eqref{stime eta pt1} and \eqref{estimeta}:
	\begin{align}\label{partial H}
		\abs{\partial_{\rho}H}=&\Bigg|\partial_\rho L+\frac{\partial_\rho L}{1+k^2\eta_{\rho}^2}-\frac{2k^2L\eta_{\rho}\eta_{\rho\rho}}{(1+k^2\eta_{\rho}^2)^2}+\frac{Ck^2\partial_{\rho}\eta_{ANG}\eta_{ANG}}{1+k^2\eta_{\rho}^2}-\frac{Ck^4\eta_{\rho}\eta_{\rho\rho}\eta_{ANG}\eta_{ANG}}{(1+k^2\eta_{\rho}^2)^2}\\
		&+\frac{2k^2\eta_{\rho\rho\rho}\eta_{ll}}{(1+k^2\eta_{\rho}^2)^2}+\frac{2k^2\eta_{\rho\rho}\partial_\rho\eta_{ll}}{(1+k^2\eta_{\rho}^2)^2}-\frac{8k^4\eta_{\rho}\eta_{\rho\rho}^2\eta_{ll}}{(1+k^2\eta_{\rho}^2)^3}\notag\Bigg|\\
		\leq&C+C+C\Theta+C\Theta+C\Theta+\frac{C}{k^2}\Theta+\frac{C}{k^2}\Theta+\frac{C}{k^2}\Theta\notag\\
		\leq& C\pa{1+\Theta};\notag\\
		\abs{\partial_{i}H}=&\abs{\partial_iL+\frac{\partial_iL}{1+k^2\eta_{\rho}^2}+\frac{C k^2\partial_i\eta_{ANG}\eta_{ANG}}{1+k^2\eta_{\rho}^2}+\frac{2k^2\eta_{\rho\rho}\partial_i\eta_{ll}}{(1+k^2\eta_{\rho}^2)^2}}\notag\\
		\leq&C\pa{1+\frac{1}{k^2}\Theta}.\notag
	\end{align}
	Note that using \eqref{eta rho rho e eta rho}, \eqref{eta ang ed eta rho}, and the fact that $\abs{\eta_{ij}}\leq Cr\abs{\eta_{\rho}}$ we deduce another estimate for $\abs{\partial_{\rho}H}$, which we will use in the final inequality of $\abs{\ol{\nabla}_i\ol{\nabla}_j\ol{S}}$:
	\begin{align}\label{partial rho H 2}
		\abs{\partial_{\rho}H}\leq C+\frac{C}{k^2}\Theta+\frac{C}{r-\rho}.
	\end{align}
	Therefore, by
	\eqref{partial H} we obtain
	\begin{align}\label{nab scal}
		\abs{\nabla_{\rho}H}=\abs{\partial_\rho H}\leq C+C\Theta;\qquad
		\abs{\nabla_i H}=\abs{\partial_i H}\leq C+\frac{C}{k^2}\Theta
	\end{align}
	and by \eqref{partial rho H 2} 
	\begin{align}\label{nabla H rho 2}
		\abs{\nabla_{\rho}H}\leq C+\frac{C}{r-\rho}+\frac{C}{k^2}\Theta.
	\end{align}
	We now compute $\partial_{\alpha}\partial_{\beta}H$, which we will later need to compute $\hess(H)$ and $\Delta H$. Towards this aim we use \eqref{stime eta pt1} and \eqref{estimeta} in $\abs{\partial_{\rho}\partial_{\rho}H}$ and \eqref{eta rho rho e eta rho}, \eqref{eta ang ed eta rho} and $\abs{\eta_{ij}}\leq Cr\abs{\eta_{\rho}},\, \abs{\partial_l\eta_{ij}}\leq Cr\abs{\eta_{\rho}}$ in the remaining terms:
	\begin{align}\label{der parziale seconda H}
		\abs{\partial_{\rho}\partial_{\rho}H}=&\Bigg|\partial_\rho\partial_{\rho}L+\frac{\partial_\rho\partial_{\rho}L}{1+k^2\eta_{\rho}^2}-\frac{Ck^2\eta_{\rho}\eta_{\rho\rho}\partial_\rho L}{(1+k^2\eta_{\rho}^2)^2}-\frac{2k^2L\eta_{\rho\rho}^2}{(1+k^2\eta_{\rho}^2)^2}-\frac{2k^2L\eta_{\rho\rho\rho}\eta_{\rho}}{(1+k^2\eta_{\rho}^2)^2}\\
		&+\frac{8k^4\eta_{\rho}^2\eta_{\rho\rho}^2L}{(1+k^2\eta_{\rho}^2)^3}+\frac{Ck^2\partial_\rho\partial_\rho\eta_{ANG}\eta_{ANG}}{1+k^2\eta_{\rho}^2}+\frac{Ck^2\partial_\rho\eta_{ANG}\partial_{\rho}\eta_{ANG}}{1+k^2\eta_{\rho}^2}\notag\\
		&-\frac{Ck^4\eta_{\rho}\eta_{\rho\rho}\partial_{\rho}\eta_{ANG}\eta_{ANG}}{(1+k^2\eta_{\rho}^2)^2}-\frac{Ck^4\eta_{\rho\rho}^2\eta_{ANG}\eta_{ANG}}{(1+k^2\eta_{\rho}^2)^2}-\frac{Ck^4\eta_{\rho}\eta_{\rho\rho\rho}\eta_{ANG}\eta_{ANG}}{(1+k^2\eta_{\rho}^2)^2}\notag\\
		&+\frac{Ck^6\eta_{\rho}^2\eta_{\rho\rho}^2\eta_{ANG}\eta_{ANG}}{(1+k^2\eta_{\rho}^2)^3}+\frac{2k^2\eta_{\rho\rho\rho\rho}\eta_{ll}}{(1+k^2\eta_{\rho}^2)^2}+\frac{2k^2\eta_{\rho\rho\rho}\partial_{\rho}\eta_{ll}}{(1+k^2\eta_{\rho}^2)^2}-\frac{8k^4\eta_{\rho}\eta_{\rho\rho}\eta_{\rho\rho\rho}\eta_{ll}}{(1+k^2\eta_{\rho}^2)^3}\notag\\
		&+\frac{2k^2\eta_{\rho\rho}\partial_{\rho}\partial_{\rho}\eta_{ll}}{(1+k^2\eta_{\rho}^2)^2}+\frac{2k^2\eta_{\rho\rho\rho}\partial_\rho\eta_{ll}}{(1+k^2\eta_{\rho}^2)^2}-\frac{8k^4\eta_{\rho}\eta_{\rho\rho}^2\partial_{\rho}\eta_{ll}}{(1+k^2\eta_{\rho}^2)^3}	-\frac{8k^4\eta_{\rho\rho}^3\eta_{ll}}{(1+k^2\eta_{\rho}^2)^3}\notag\\
		&-\frac{16k^4\eta_{\rho\rho\rho}\eta_{\rho\rho}\eta_{\rho}\eta_{ll}}{(1+k^2\eta_{\rho}^2)^3}-\frac{8k^4\eta_{\rho}\eta_{\rho\rho}^2\partial_{\rho}\eta_{ll}}{(1+k^2\eta_{\rho}^2)^3}+\frac{48k^6\eta_{\rho}^2\eta_{\rho\rho}^3\eta_{ll}}{(1+k^2\eta_{\rho}^2)^4}
		\Bigg|\notag\\
		\leq& C+C+\frac{C}{k^2}\Theta+\frac{C}{k^2}\Theta+\frac{C}{k^2}\Theta+\frac{C}{k^2}\Theta+C\Theta+C\Theta+C\Theta+C\Theta+C\Theta+C\Theta\notag\\
		&+\frac{C}{k^2}\Theta+\frac{C}{k^2}\Theta+\frac{C}{k^2}\Theta+\frac{C}{k^2}\Theta+\frac{C}{k^2}\Theta+\frac{C}{k^2}\Theta+\frac{C}{k^2}\Theta+\frac{C}{k^2}\Theta+\frac{C}{k^2}\Theta+\frac{C}{k^2}\Theta\notag\\
		\leq& C+\frac{C}{k^2}\Theta+C\Theta;\notag\\
		\abs{\partial_l\partial_{\rho}H}=&\Bigg|\partial_l\partial_\rho L+\frac{\partial_l\partial_\rho L}{1+k^2\eta_{\rho}^2}-\frac{2k^2\partial_lL\eta_{\rho}\eta_{\rho\rho}}{(1+k^2\eta_{\rho}^2)^2}+\frac{Ck^2\partial_l\partial_{\rho}\eta_{ANG}\eta_{ANG}}{1+k^2\eta_{\rho}^2}+\frac{Ck^2\partial_{\rho}\eta_{ANG}\partial_l\eta_{ANG}}{1+k^2\eta_{\rho}^2}\notag\\&-\frac{Ck^4\eta_{\rho}\eta_{\rho\rho}\partial_l\eta_{ANG}\eta_{ANG}}{(1+k^2\eta_{\rho}^2)^2}+\frac{2k^2\eta_{\rho\rho\rho}\partial_{l}\eta_{tt}}{(1+k^2\eta_{\rho}^2)^2}+\frac{2k^2\eta_{\rho\rho}\partial_l\partial_\rho\eta_{tt}}{(1+k^2\eta_{\rho}^2)^2}-\frac{8k^4\eta_{\rho}\eta_{\rho\rho}^2\partial_l\eta_{tt}}{(1+k^2\eta_{\rho}^2)^3}\notag\Bigg|\notag\\
		\leq&C+C+\frac{C}{k^2}\Theta+\frac{C}{r-\rho}+\frac{C}{r-\rho}+\frac{C}{r-\rho}+\frac{C}{k^2}\Theta+\frac{C}{k^2}\Theta+\frac{C}{k^2}\Theta\notag\\
		\leq&C+\frac{C}{k^2}\Theta+\frac{C}{r-\rho};\notag\\
		\abs{\partial_t\partial_i H}=&\abs{\partial_t\partial_iL+\frac{\partial_t\partial_iL}{1+k^2\eta_{\rho}^2}+\frac{C k^2\partial_t\partial_i\eta_{ANG}\eta_{ANG}}{1+k^2\eta_{\rho}^2}+\frac{C k^2\partial_i\eta_{ANG}\partial_t\eta_{ANG}}{1+k^2\eta_{\rho}^2}+\frac{2k^2\eta_{\rho\rho}\partial_t\partial_i\eta_{ll}}{(1+k^2\eta_{\rho}^2)^2}}\notag\\
		&\leq C+C+C+C+\frac{C}{k^2}\Theta\notag\\
		\leq& C+\frac{C}{k^2}\Theta.\notag\notag
	\end{align}
	It follows by \eqref{der parziale seconda H},
	\begin{align}\label{hess H}
		\abs{\nabla_{\rho}\nabla_{\rho}H}=&\abs{\partial_\rho\partial_\rho H}\leq C\pa{1+\Theta}\notag\\
		\abs{\nabla_\rho\nabla_lH}=&\abs{\nabla_l\nabla_{\rho}H}=\abs{\partial_l\partial_{\rho}H-\Gamma_{l\rho}^t\nabla_tH}\notag\\
		\leq&\abs{\partial_l\partial_{\rho}H}+\abs{\Gamma_{l\rho}^t\nabla_tH}\notag\\
		\leq&C\pa{1+\frac{1}{r-\rho}+\frac{1}{k^2}\Theta}+C\rho\pa{1+\frac{1}{k^2}\Theta}\leq C\pa{1+\frac{1}{r-\rho}+\frac{1}{k^2}\Theta};\\
		\abs{\nabla_t\nabla_i H}=&\abs{\partial_t\partial_i H-\Gamma^\alpha_{ti}\nabla_\alpha H}=\abs{\partial_t\partial_i H-\Gamma^\rho_{ti}\nabla_\rho H-\Gamma^l_{ti}\nabla_l H}\notag\\
		\leq&C\pa{1+\frac{1}{k^2}\Theta}+C\rho\pa{1+\frac{1}{r-\rho}+\frac{1}{k^2}\Theta}+\rho\pa{1+\frac{1}{k^2}\Theta}\notag\\
		\leq&C\pa{1+\frac{1}{r-\rho}+\frac{1}{k^2}\Theta}\notag
	\end{align}
	where in the last two inequality we have used \eqref{stime christoffel} and \eqref{nabla H rho 2}.
	The Hessian of $\ol{S}$ is given by
	\begin{align*}
		\ol{\nabla}_\alpha\ol{\nabla}_\beta\ol{S}=\nabla_{\alpha}\nabla_\beta S+\nabla_{\alpha}\nabla_{\beta}H-G^\tau_{\alpha\beta}\nabla_{\tau}S+G^\tau_{\alpha\beta}\nabla_{\tau}H,
	\end{align*}
	then, to get an estimate on its norm with respect to the metric $\ol{g}$, we need to compute
	\begin{align*}
		\abs{\ol{\nabla}_{\rho}\ol{\nabla}_{\rho}\ol{S}}_{\ol{g}};\quad \abs{\ol{\nabla}_i\ol{\nabla}_{\rho}\ol{S}}_{\ol{g}};\quad \abs{\ol{\nabla}_i\ol{\nabla}_j\ol{S}}_{\ol{g}},
	\end{align*} 
	that are
	\begin{align*}
		\abs{\ol{\nabla}_\rho\ol{\nabla}_\rho\ol{S}}_{\ol{g}}&=\frac{1}{1+k^2\eta_{\rho}}\abs{\ol{\nabla}_\rho\ol{\nabla}_\rho\ol{S}}_g=\frac{1}{1+k^2\eta_{\rho}}\abs{\ol{\nabla}_\rho\ol{\nabla}_\rho\ol{S}}\\
		&=\frac{1}{(1+k^2\eta_{\rho})^2}\abs{\nabla_{\rho}\nabla_{\rho}S+\nabla_{\rho}\nabla_{\rho}H-G^\rho_{\rho\rho}\nabla_{\tau}S+G^\rho_{\rho\rho}\nabla_{\rho}H}\\
		&\leq \frac{1}{1+k^2\eta_{\rho}}\sq{C+C\Theta+C\Theta+C\Theta}\\
		&\leq C\pa{1 + \frac{1}{k^2}\Theta};\\
		\abs{\ol{\nabla}_i\ol{\nabla}_\rho\ol{S}}_{\ol{g}}&=\frac{1}{(1+k^2\eta_{\rho}^2)^{\frac{1}{2}}}\abs{\ol{\nabla}_i\ol{\nabla}_{\rho}\ol{S}}_g\leq\frac{C}{(1+k^2\eta_{\rho}^2)^{\frac{1}{2}}}\abs{\ol{\nabla}_i\ol{\nabla}_{\rho}\ol{S}}\\
		&=\frac{1}{(1+k^2\eta_{\rho}^2)^{\frac{1}{2}}}\abs{\nabla_{i}\nabla_\rho S+\nabla_{i}\nabla_{\rho}H-G^\tau_{i\rho}\nabla_{\tau}S+G^\tau_{i\rho}\nabla_{\tau}H}\\
		&=\frac{1}{(1+k^2\eta_{\rho}^2)^{\frac{1}{2}}}\abs{\nabla_{i}\nabla_\rho S+\nabla_{i}\nabla_{\rho}H}\\
		&\leq\frac{1}{(1+k^2\eta_{\rho}^2)^{\frac{1}{2}}}\sq{C\pa{1+\frac{1}{k^2}\Theta}+C\pa{1+\frac{1}{r-\rho}+\frac{1}{k^2}\Theta}}\\
		&\leq C\pa{1+\frac{1}{r-\rho}+\frac{1}{k^2}\Theta};\\
		\abs{\ol{\nabla}_i\ol{\nabla}_j\ol{S}}_{\ol{g}}&= \abs{\ol{\nabla}_i\ol{\nabla}_j\ol{S}}_g\leq C\abs{\ol{\nabla}_i\ol{\nabla}_j\ol{S}}\\
		&=\abs{\nabla_{i}\nabla_j S+\nabla_{i}\nabla_{j}H-G^\rho_{ij}\nabla_{\rho}S+G^\rho_{ij}\nabla_{\rho}H}\\
		&\leq C+C\pa{1+\frac{1}{r-\rho}+\frac{1}{k^2}\Theta}+C\rho+C\pa{1+\frac{1}{r-\rho}+\frac{1}{k^2}\Theta}\\
		&\leq C\pa{1+\frac{1}{r-\rho}+\frac{1}{k^2}\Theta},
	\end{align*}
	where we have used equations \eqref{normgbarg}, \eqref{G}, \eqref{nab scal} and \eqref{hess H} in the estimate of $\abs{\ol{\nabla}_\rho\ol{\nabla}_\rho\ol{S}}_{\gb}$ and \eqref{normgbarg}, \eqref{G}, \eqref{nabla H rho 2} and \eqref{hess H} in the the estimates of $\abs{\ol{\nabla}_i\ol{\nabla}_\rho\ol{S}}_{\ol{g}}$ and $\abs{\ol{\nabla}_i\ol{\nabla}_j\ol{S}}_g$.
	Therefore, we obtain
	\begin{equation}\label{hessS}
		\abs{\ol{\hess}\,\ol{S}}_{\gb}=\abs{\ol{\nabla}_{\rho}\ol{\nabla}_{\rho}\ol{S}}_{\ol{g}}+\abs{\ol{\nabla}_{\rho}\ol{\nabla}_i\ol{S}}_{\ol{g}}+\abs{\ol{\nabla}_{i}\ol{\nabla}_{\rho}\ol{S}}_{\ol{g}}+\abs{\ol{\nabla}_i\ol{\nabla}_j\ol{S}}_{\ol{g}}\leq C\pa{1+\frac{1}{r-\rho}+\frac{1}{k^2}\Theta}
	\end{equation}
	and, by \eqref{hess H} and the definition of $\gb$, we also have
	\begin{align}\label{lapS}
		\abs{\bar{\Delta}\ol{S}}_{\gb}=\abs{\ol{g}^{\alpha\beta}\ol{\nabla}_{\alpha}\ol{\nabla}_{\beta}\ol{S}}_{\gb}=\abs{\gb^{\rho\rho}\ol{\nabla}_{\rho}\ol{\nabla}_{\rho}\ol{S}+\gb^{ti}\ol{\nabla}_{t}\ol{\nabla}_{i}\ol{S}}_{\gb}\leq C\pa{1+\frac{1}{r-\rho}+\frac{1}{k^2}\Theta}.
	\end{align}
	As a consequence, putting together \eqref{SRic}, \eqref{bach pezzi}, \eqref{lapRic}, \eqref{hessS} and \eqref{lapS}, we deduce:
	\begin{equation}\label{stima barbach}
		\abs{\ol{\bach}}_{\gb}\leq C\pa{1+\frac{1}{r-\rho}+\frac{1}{k^2}\Theta},
	\end{equation}
	which implies \eqref{estbach}.

\

	\section*{Acknowledgements}
	The authors would like to thank Professor P. Mastrolia for
	his useful help regarding the computations in Section
	\ref{app} and his suggestions about the whole manuscript. 
	All three authors are members of the Gruppo Nazionale per le Strutture Algebriche, Geometriche e loro
	Applicazioni (GNSAGA) of INdAM (Istituto Nazionale di Alta Matematica).

\
	
	\bibliographystyle{abbrv}
	\bibliography{biblio_smallBach}
\end{document}